\documentclass[10pt,reqno]{amsart}
\usepackage{amsmath,amsthm,amsfonts,eucal,enumerate,color}
\usepackage{latexsym}
\usepackage{amssymb}
\usepackage{hyperref}
\renewcommand{\thesection}{\arabic{section}}

\newtheorem{theorem}{Theorem}[section]

\newtheorem{lemma}[theorem]{Lemma}
\newtheorem{prop}[theorem]{Proposition}
\newtheorem{defi}[theorem]{Definition}

\usepackage{fullpage}

\renewcommand{\theequation}{\thesection .\arabic{equation}}
\let\subs\subsection
\renewcommand\subsection{\setcounter{equation}{0}
\gdef\theequation{\thesubsection \arabic{equation}}\subs}
\let\sect\section
\renewcommand\section{\setcounter{equation}{0}
\gdef\theequation{\thesection .\arabic{equation}}\sect}

\newcommand{\bbR}{{\mathbb{R}}}
\newcommand{\bbT}{\mathbb{T}}
\newcommand{\bbC}{\mathbb{C}}
\newcommand{\bbZ}{\mathbb{Z}}
\newcommand{\bbD}{\mathbb{D}}
\newcommand{\bbE}{\mathbb{E}}

\newcommand{\cA}{{\mathcal{A}}}
\newcommand{\cB}{{\mathcal{B}}}

\newcommand{\cD}{{\mathcal{D}}}

\newcommand{\cR}{{\mathcal{R}}}

\newcommand{\cM}{{\mathcal{M}}}

\newcommand{\cS}{{\mathcal{S}}}

\newcommand{\cP}{{\mathcal{P}}}

\newcommand{\IR}{{\mathbb{R}}}

\newcommand{\IZ}{{\mathbb{Z}}}
\newcommand{\zv}{\IZ^\nu}

\newcommand{\be}{\begin{equation}}
\newcommand{\ee}{\end{equation}}

\newcommand{\dist}{\mathop{\rm{dist}}}

\newcommand{\sgn}{\mathop{\rm{sgn}}}

\newcommand{\ve}{\varepsilon}

\def\zero{{(0)}}

\renewcommand{\Re}{\mathop{\mathrm{Re}}}
\renewcommand{\Im}{\mathop{\mathrm{Im}}}

\newcommand{\ol}[1]{\overline{#1}}

\newcommand{\bbN}{\mathbb{N}}
\newcommand{\bbQ}{\mathbb{Q}}

\begin{document}

\title{Comb domains of Schr\"odinger operators with small quasiperiodic potentials}

\author{Ilia Binder}

\address{Department of Mathematics, University of Toronto, Bahen Centre, 40 St. George St., Toronto, Ontario, CANADA M5S 2E4}

\email{ilia@math.toronto.edu}

\thanks{I.\ B.\ was supported in part by an NSERC Discovery Grant}

\author{David Damanik}

\address{Department of Mathematics, Rice University, 6100 S. Main St. Houston TX 77005-1892, U.S.A.}

\email{damanik@rice.edu}

\thanks{D.\ D.\ was supported in part by NSF grants DMS--2054752 and DMS--2349919}

\author{Michael Goldstein}

\address{Department of Mathematics, University of Toronto, Bahen Centre, 40 St. George St., Toronto, Ontario, CANADA M5S 2E4}

\email{gold@math.toronto.edu}

\author{Milivoje Luki\'c}

\address{Department of Mathematics, Rice University, 6100 S. Main St. Houston TX 77005-1892, U.S.A.}

\email{milivoje.lukic@rice.edu}

\thanks{M.\ L.\ was supported in part by NSF grant DMS--2154563}

\begin{abstract}
We characterize spectra of Schr\"odinger operators with small quasiperiodic analytic potentials in terms of their comb domains, and study action variables motivated by the KdV integrable system.
\end{abstract}

\maketitle

\vspace{-10pt}
\begin{center}
	{\it  Dedicated to Michael Goldstein}
\end{center}

\tableofcontents

\section{Introduction}

For a one-dimensional Schr\"odinger operator
\begin{equation} \label{eq:SLPAI1-1}
[H_V u](x) = - u''(x) +  V(x) u(x), \quad x \in \IR
\end{equation}
with an almost periodic potential $V$, by gap labelling theory, the spectrum can be expressed in the form
\begin{equation}\label{eq:1spectrum}
\cS = [E_0 , \infty) \setminus \bigcup_{k \in \cM \cap (0,\infty) } (E^-_k, E^+_k),
\end{equation}
where $\cM$ is the frequency module of $V$ and $E_0 < E^-_k \le E^+_k < E^-_l$ whenever $k,l \in \cM$ with $0 < k < l$; compare Johnson-Moser \cite{JM} as well as \cite{DF22, DF23, J86}.  The $k$-th gap is said to be open when  $E^-_k < E^+_k$. If $V$ is periodic with period $T$, then $\cM = T^{-1} \bbZ$, which results in the band/gap structure of the spectrum. In this case, the connection between the sequence of gap lengths and smoothness of the potential is akin to a nonlinear Fourier transform \cite{Hochstadt,MO,Tr,Poschel}: the decay rate of gap lengths encodes the smoothness of the potential. If $V$ is not periodic, $\cM$ is dense in $\bbR$, and Cantor spectrum is common \cite{Eliasson, FJP02, Moser}; in fact, the spectrum of an almost periodic operator can have zero Lebesgue measure, and even zero Hausdorff dimension \cite{DFL}. Similar results, and in fact more refined and detailed ones, are known for discrete Schr\"odinger operators with almost periodic potentials; compare \cite{Avila09, ABD09, ABD12, AJ09, AJ10, ALSQ24, FV, GJY, GS11, HHSY25, P04, WZ17}. 

Conversely, it is of great interest to establish strong upper bounds on gap sizes when possible, with decay with respect to an appropriate indexing of $\cM \cap (0,\infty)$. For a class of limit-periodic potentials exponentially well approximated by periodic potentials, upper bounds were established by Pastur--Tkachenko \cite{PasturTkachenko1,PasturTkachenko2}, see also \cite{Chulaevskii} for this class of potentials.
In the quasiperiodic setting, Damanik--Goldstein \cite{DG1} derived exponentially decaying upper bounds for small analytic quasiperiodic operators with Diophantine frequencies, by a multiscale analysis on the dual group. This work led to thickness properties of the spectrum such as homogeneity and Craig-type conditions \cite{DGL3}, which in turn led to a description of the isospectral torus \cite{DGL2,DGL1} and applications to the KdV equation for small quasiperiodic initial data \cite{BDGL1, DG2}. See also related work on gap sizes and homogeneity \cite{DGSV,LYZZ} and on the KdV equation and the Toda flow \cite{BDLV18, ChapoutoKillipVisan, EVY, LukicYoung, VinnikovYuditskii, Tsugawa}.

The central object of this paper are Marchenko--Ostrovskii maps \cite{MO}, originally introduced in the periodic setting in order to provide a parametrization of spectra of periodic Schr\"odinger operators. Let us recall a more general definition given by Johnson--Moser \cite{JM}. 
Denote by  $m(x;z)$ the Weyl function of the half-line operator $-\partial_x^2+V$ on $[x,\infty)$, and denote $\bbC_+ = \{z \in\bbC \mid \Im z > 0 \}$. If $V$ is almost periodic, Johnson--Moser \cite{JM} proved that $m(x,z)$ is an almost periodic function of $x$ for any $z\in \bbC_+$, and introduced its mean value
\begin{equation}\label{wfunc}
w(z) = \bbE( m(x;z)) =  \lim_{X\to\infty} \frac 1X \int_0^X m(x;z) \,dx
\end{equation}
(in this paper, $\bbE$ always denotes the mean value of an almost periodic function of $x\in \bbR$, as above). 
The Marchenko--Ostrovskii map $w$ is a fundamental object of interest, since $-\Re w$ is the Lyapunov exponent, $\Im w$ is the rotation number, and $w$ is an integral transform of the density of states of $H_V$; the resulting connection between the Lyapunov exponent and density of states is the celebrated Thouless formula \cite{PasturFigotin}. We note also that a further generalization of  Marchenko--Ostrovskii maps to ergodic operators was studied by Kotani \cite{Kotani}, and that there is a corresponding theory even beyond the ergodic setting, which first appeared as the theory of Stahl--Totik regularity \cite{StahlTotik} in the orthogonal polynomial setting, and was extended to the  Schr\"odinger operator setting in \cite{EL}.

The formula \eqref{wfunc} defines $w$ on the set $\bbC \setminus [E_0, \infty)$, and $w$ is shown to be injective on this set, so it is a conformal bijection; we follow a standard convention to rotate the image clockwise and call the set
\[
\Omega = -iw( \bbC \setminus  [E_0, \infty))
\]
 the Marchenko--Ostrovskii domain corresponding to $V$. This provides a powerful connection between spectral theory and geometric function theory. On the set $\bbC \setminus [E_0, \infty)$, the function obeys the symmetry $\ol{w(\ol z)} = w(z)$ and the asymptotics
\begin{equation}\label{eqn2term}
w(z) = - \sqrt{-z} - \frac{\bbE(V)}{2\sqrt{-z}} + o( \lvert z \rvert^{-1/2} ), \qquad z \to -\infty,
\end{equation}
so the domain $\Omega$ determines the map $w$ up to an additive constant; if $w$ corresponds to $\Omega$, so does $w(\cdot - c)$, for $c \in\bbR$. Since $w$ determines the density of states measure, and the support of this measure is the spectrum $\cS$, the domain $\Omega$ also determines $\cS$ uniquely up to a constant. The constant corresponds to a constant shift in the potential and to a translation of the spectrum, and it can be fixed by setting $\bbE(V)=0$ (common in direct spectral theory \cite{DGL2,KP}) or by setting $E_0 = 0$ (common in inverse spectral theory \cite{EVY}).

In the periodic setting, the domain $\Omega$ is a comb domain of the form $\bbC_+ \setminus \bigcup_{n\in\bbZ \setminus \{0\}} (n/T, n/T + i h_n]$, with $h_{-n} = h_n \ge 0$. The slits $(n/T, n/T + i h_n]$ correspond to gaps of the spectrum, and Marchenko--Ostrovskii proved that Sobolev space conditions $V\in H^k(\bbT)$ are equivalent to weighted $\ell^2$ conditions on the slits $\sum_n (n^{k+1} h_n)^2 < \infty$. This has inspired further work in periodic settings \cite{Korotyaev,Tkachenko} and beyond; see \cite{EY,HurRemling,EremenkoSodin}. Beyond the periodic setting,  comb domains have also appeared in the characterization of a class of limit periodic potentials,  a class of reflectionless operators on homogeneous spectra \cite{PasturTkachenko2,SY1,SY2,GesztesyYuditskii}, and a class of limit periodic operators associated with quadratic iterations \cite{SodinYuditskii90}.

When the Lyapunov exponent is zero on the spectrum, the harmonic function $-\Re w$ on $\bbC \setminus \cS$ is precisely the Martin function for the Denjoy domain $\bbC \setminus \cS$ normalized by \eqref{eqn2term}. This connection between Martin functions and reflectionless operators plays a prominent role in the papers of Yuditskii and coauthors \cite{SY2,Yuditskii,DamanikYuditskii,EVY,KheifetsYuditskii,BLY2}.

In the first part of this paper, we characterize the Marchenko--Ostrovskii domains corresponding to small quasiperiodic operators with Diophantine frequencies. These are comb domains, and we show how the structure of the comb domain encodes the smallness and analyticity of the quasiperiodic potential. To proceed, we must precisely define the class of operators.

\begin{defi}
Let $0 < \kappa_0 \le 1$ and $\omega \in \mathbb{R}^\nu$ for some $\nu \in \mathbb{N}$. We denote by $\cB(\omega,\kappa_0)$ the Banach space of functions  $V: \mathbb{R} \to \mathbb{R}$  of the form
\begin{equation}\label{eq:samplingfunc}
V(x) = \sum_{n \in \zv} c(n) e^{2 \pi i n\omega x}\ 
\end{equation}
for which the norm
\begin{equation}\label{eq:weightedellinfty}
\lVert V \rVert_{\infty,\kappa_0}  = \sup_{n\in\zv} \lvert c(n) \rvert \exp( \kappa_0 \lvert n\rvert)
\end{equation}
is finite.
In particular, for $\ve> 0$, we denote by $\cP(\omega,\ve,\kappa_0)$ the set of $V \in \cB(\omega,\kappa_0)$ with $\lVert V \rVert_{\infty,\kappa_0} \le \ve$.
\end{defi}

It has been common to view $\mathcal{P}(\omega,\ve,\kappa_0)$ a metric space with the metric inherited from $L^\infty(\mathbb{R})$, but the weighted $\ell^\infty$ norm \eqref{eq:weightedellinfty} will be useful below. 
With this norm, the interior of $\cP(\omega,\ve,\kappa_0)$ is the set $\mathring\cP(\omega,\ve,\kappa_0)$ given by $\lVert V \rVert_{\infty,\kappa_0}  < \ve$.

We will assume that $\omega = (\omega_1, \dots, \omega_\nu) \in \mathbb{R}^\nu$ obeys the Diophantine condition
\begin{equation}\label{eq:1PAI7-5-85a}
|n \omega| \ge a_0 |n|^{-b_0}, \quad n \in \mathbb{Z}^\nu \setminus \{ 0 \}
\end{equation}
for some
\begin{equation}\label{eq:PAIombasicTcondition5a}
0 < a_0 < 1,\quad \nu < b_0 < \infty.
\end{equation}
In \eqref{eq:1PAI7-5-85a}, and everywhere else in this work, we use the $\ell^1$ norm on $\bbZ^\nu$:
\[
|n| = |(n_1, \dots, n_\nu)| = \sum_j |n_j|, \qquad n \in \zv.
\]

Potentials $V\in \cP(\omega,\ve,\kappa_0)$ were studied in \cite{DG1} via a multiscale analysis method, which used smallness, analyticity, and Diophantine frequency to derive much more precise description of the gap sizes and their distances. An extension of this multiscale analysis to Abelian groups \cite{DGL1}, combined with periodic approximation, was used to describe the isospectral torus of $V \in \cP(\omega,\ve,\kappa_0)$ \cite{DGL2}, to show that the corresponding spectrum is homogeneous \cite{DGL3}, and to study the KdV equation with this class of initial data \cite{BDGL1}.  

We will apply this multiscale analysis and periodic approximation approach to describe the comb domains:

\begin{theorem}\label{thmCombDirect} 
Assume that $\omega$ obeys the Diophantine condition
\eqref{eq:1PAI7-5-85a}, \eqref{eq:PAIombasicTcondition5a}. There exists $\ve_1 = \ve_1(a_0,b_0,\kappa_0) > 0$ such that, if $\ve < \ve_1$ and $V \in \mathcal{P}(\omega,\ve,\kappa_0)$, the corresponding Marchenko--Ostrovskii domain is a comb domain of the form
\begin{equation}\label{eqncombdomain1}
\bbC_+ \setminus  \bigcup_{ m \in \mathbb{Z}^\nu \setminus\{0\}} (m\omega, m\omega + ih_m]
\end{equation}
and the slit sizes $h_m$ obey $h_{-m} = h_m \ge 0$ and 
\[
h_m \le \varepsilon^{1/2} \exp\left( - \frac{\kappa_0}5 \lvert m \rvert \right).
\]
Moreover, each $h_m$ is a continuous function of $V \in \cP(\omega,\ve,\kappa_0)$.
\end{theorem}

Recall further that the isospectral torus $\cR(\cS)$ is defined as the set of reflectionless operators with spectrum $\cS$, and that the torus of Dirichlet data is the product of circles obtained as double covers of the gap closures,
\[
\cD(\cS) = \prod_{\substack{ k \in \cM \cap (0,\infty) \\ E_k^- < E_k^+}} C_k, \quad C_k =  ([E_k^-, E_k^+] \times \{-1,+1\} ) /_{(E_k^\pm, -1) \sim (E_k^\pm,+1)}.
\]
If $\cS$ is homogeneous and has finite gap length, by Sodin--Yudiskii \cite{SY2,SY1}, the isospectral torus is parametrized by the Dirichlet data and consists of almost periodic Schr\"odinger operators. A priori, the frequency module can be read off from the set of bases of the slits; in particular, quasiperiodicity with frequency $\omega$ can be read off from \eqref{eqncombdomain1}. However, quantitative information about the analyticity and smallness of the sampling function is more specialized, and this is the subject of the following theorem:

\begin{theorem}\label{thmCombInverse}
Assume that $\omega$ obeys the Diophantine condition
\eqref{eq:1PAI7-5-85a}, \eqref{eq:PAIombasicTcondition5a}. There exists $\ve^{(0)}(a_0,b_0,\kappa_0) > 0$ such that the following holds. Let $h_m \ge 0$, indexed by $m\in \mathbb{Z}^\nu \setminus \{0\}$, obey the conditions $h_m = h_{-m}$ and 
\[
h_m \le \ve' \exp\left( - \kappa \lvert m\rvert \right)
\]
with $\ve' < \ve^{(0)}$ and $\kappa \ge 5 \kappa_0$. Let $\cS$ be the corresponding spectrum with $\min \cS = 0$.  For every $\mu \in \cD(\cS)$, there is a corresponding potential $V \in \cR(\cS)$ with $V \in \cP(\omega,(\ve')^{1/4},  {\kappa}/ 3)$.

In particular, $\cR(\cS) \subset \mathcal{P}(\omega,(\ve')^{1/4},  {\kappa}/ 3)$.
\end{theorem}

The proof consists of constructing, for any Dirichlet data, a quasiperiodic potential $V \in \mathcal{P}(\omega,(\ve')^{1/4},\frac {\kappa}3)$ with the desired comb domain and Dirichlet data. The explicit constants in these theorems are not optimized.

In the second part of the paper, we derive some differentiability results motivated by the inverse spectral theory of periodic operators. By general principles, gap edges are $1$-Lipshitz with respect to the $L^\infty$-norm. We will consider differentiability of gap edges with respect to the potential, expressed in the language of Fr\'echet derivatives; for general background see \cite[Appendix A]{PT} and note that $\mathring\cP(\omega,\ve,\kappa_0)$ is an open ball in the Banach space $\cB(\omega,\kappa_0)$.  Gap length cannot be a $C^1$ function on $\mathring\cP(\omega,\ve,\kappa_0)$; it is not differentiable even in the periodic case. However, the following related quantities are:

\begin{theorem}\label{thmGaps}
Assume that $\omega$ obeys the Diophantine condition
\eqref{eq:1PAI7-5-85a}, \eqref{eq:PAIombasicTcondition5a}. There exists $\ve_1 = \ve_1(a_0,b_0,\kappa_0) > 0$ such that, for $m\in \zv$ with $m\omega > 0$, the functions
\begin{equation}\label{eqnGapMidpoint}
\gamma_m^2 := (E_m^+ - E_m^-)^2, \qquad \tau_m := \frac{E_m^- + E_m^+}2
\end{equation}
are $C^1$ functions on $\mathring\cP(\omega,\ve_1,\kappa_0)$.
\end{theorem}

The proof will also provide explicit expressions for their derivatives, in terms of quasiperiodic eigensolutions of the Schr\"odinger operators at the gap edges, which mirror known formulas in the periodic setting. The periodic problem is closely related to the Dirichlet problem on an interval, and gap edges can be interpreted as isolated eigenvalues of operators with periodic/antiperiodic boundary conditions, so their differentiability as a function of the potential can be obtained using perturbation theory. In the quasiperiodic setting, there is no such interpretation; thus, the proof of Theorem~\ref{thmGaps} uses the results of the periodic theory, rather than its techniques.

We will also study a set of actions on the open set of quasiperiodic potentials $\mathring\cP(\omega,\ve,\kappa_0)$, which we define as follows: for $V\in \mathcal{P}(\omega,\ve,\kappa_0)$ and $m\in \zv$ with $m\omega >0$, let
\begin{equation}\label{eqnAction}
I_m := - \frac {2}{\pi}  \int_{E_m^-}^{E_m^+} \Re w(\lambda) \,d\lambda.
\end{equation}
If $E_m^- = E_m^+$, the convention $I_m = 0$ is natural.  Actions were originally introduced by Flaschka--McLaughlin \cite{FMcL} in the periodic case by an expression in terms of the discriminant, which can be rewritten in the form \eqref{eqnAction}. In the periodic setting, integrability of the KdV equation motivated the construction of action-angle coordinates for periodic initial data $V \in L^2(\bbT)$, whose Cartesian counterpart are a global, real analytic system of Birkhoff coordinates \cite{KappelerMityagin99,KP}. In these coordinates, the KdV equation is represented as an integrable Hamiltonian system on $L^2(\bbT)$, with the Hamiltonian a function of the actions alone. Joint analyticity of this set of coordinates has played an important role in further studies of the periodic KdV equation and has been the basis for applications of KAM techniques to perturbations of the periodic KdV equation \cite{KP}.

From a spectral theorist's perspective, in the periodic case, the actions describe the isospectral torus and they stay constant along the KdV time evolution, whereas the angles are a set of coordinates on the isospectral torus and they evolve linearly with time. In the reflectionless nonperiodic setting, inverse spectral theory is based on character-automorphic functions on the resolvent set of the Schr\"odinger operator. All results in this setting describe single solutions, and the set of angles is reinterpreted as an element of the character group \cite{EVY, SY2, SY1}. It is a natural question to what extent the actions \eqref{eqnAction}, combined with these angles/characters, can play the analogous role in the quasiperiodic setting, and how much of a Hamiltonian system structure can be found in this setting. Differentiability of the actions is a  step in this direction:

\begin{theorem}\label{thmActions}
Assume that $\omega$ obeys the Diophantine condition
\eqref{eq:1PAI7-5-85a}, \eqref{eq:PAIombasicTcondition5a}. There exists $\ve_1 = \ve_1(a_0,b_0,\kappa_0) > 0$ such that, for $m\in \zv$ with $m\omega > 0$, the action $I_m$ is  well-defined for $V \in \cP(\omega,\ve_1,\kappa_0)$ and has the following properties:
\begin{enumerate}[(a)]
\item $I_m \ge 0$, with $I_m = 0$ if and only if the $m$-th gap is closed
\item Actions obey the sum identity
\begin{equation}\label{eqnTraceFormula}
\sum_{\substack{m\in \zv \\  m\omega > 0}} 2 \pi m\omega I_m(V) = \frac 12 \bbE (V^2).
\end{equation}
\item  $I_m$ is a $C^1$ function on $\mathring\cP(\omega,\ve_1,\kappa_0)$, and its Fr\'echet derivative is the functional
\begin{equation}\label{imderivative}
(\partial_V I_m )(q) =  \frac 2{\pi}  \bbE \left( q(x)  \int_{E_m^-}^{E_m^+} G(x,x;\lambda,V) d\lambda \right),
\end{equation}
where $G(x,x;\lambda,V)$ denotes the diagonal Green's function of the Schr\"odinger operator $H_V$.
\item The function $I_m / \gamma_m^2$, initially defined on the set $\{ V \in \mathring\cP(\omega,\ve_1,\kappa_0)  \mid \gamma_m(V) > 0 \}$, extends to a strictly positive, $C^1$ function on $\mathring\cP(\omega,\ve_1,\kappa_0)$.
\end{enumerate}
\end{theorem}

Part (d) tells us that the action $I_m$ scales quadratically with gap size as the $m$-th gap closes. It is a  strengthening of (c); in the periodic setting, this strengthening is a necessary step in describing the Birkoff coordinates (compare \cite[Theorem 7.3]{KP}).

The proofs cannot follow the strategy from the periodic theory. The periodic theory relies heavily on the behavior along a single period and on objects, such as the monodromy matrix, which do not exist in non-periodic settings. The monodromy matrix can be defined for complex periodic potentials, so some variables are naturally found to be complex analytic \cite{KP}; this has no counterpart in our setting. Instead, in this  paper, as in \cite{DGL2}, we use periodic approximation. Let $\omega^{(r)} \to \omega$ be a canonical sequence of rational approximants for $\omega$; this sequence will be fixed in what follows. Since $\omega$ has rationally independent components, each $V \in \mathcal{P}(\omega,\ve,\kappa_0)$ corresponds to a unique continuous function $U:\mathbb{T}^\nu \to \mathbb{R}$ such that $V(x) = U(\omega x)$. This, in turn, allows us to unambiguously define
\begin{equation}\label{eqnPeriodicApproximantVr}
V^{(r)}(x) = U(\omega^{(r)} x) = \sum_{n \in \zv} c(n) e^{2 \pi i n\omega^{(r)} x}.
\end{equation}
We therefore have a map $\mathcal{P}(\omega,\ve,\kappa_0) \to \mathcal{P}(\omega^{(r)},\ve,\kappa_0)$ given by $V \mapsto V^{(r)}$.

Our method also requires approximation of analytic sampling functions by those with subexponentially decaying Fourier coefficients, replacing the condition
\begin{equation}\label{eq:samplingfunc2}
\lvert c(n)\rvert  \le  \ve\exp(-\kappa_0 \lvert n\rvert ), \quad \forall n \in \zv
\end{equation}
by
\begin{equation}\label{eq:samplingfunc2alpha}
\lvert c(n)\rvert  \le  \ve\exp(-\kappa_0 \lvert n\rvert^{\alpha_0} ), \quad \forall n \in \zv
\end{equation}
for some $\alpha_0 \in (0,1]$. The corresponding set of quasiperiodic potentials is denoted by $\mathcal{P}(\omega,\ve,\kappa_0,\alpha_0)$. Of course, $\mathcal{P}(\omega,\ve,\kappa_0,1) = \mathcal{P}(\omega,\ve,\kappa_0)$.

In Section~\ref{sectionMOcontinuity}, we consider a very general continuity property of the Marchenko--Ostrovskii maps away from the spectrum. In Section~\ref{sectionMOperiodic} we study two-sided estimates on the comb domains associated to periodic approximants $V^{(r)} \in \mathcal{P}(\omega^{(r)},\ve,\kappa_0)$. In Section~\ref{sectionGeometric} we combine this with geometric function theory to describe the comb domains of $V \in \cP(\omega,\ve,\kappa_0)$, proving Theorem~\ref{thmCombDirect}. In Section~\ref{sectionTranslation} we modify a technique from \cite{DGL2} to prove Theorem~\ref{thmCombInverse} by periodic approximation.
In Section~\ref{sectionGaps} we prove Theorem~\ref{thmGaps} and in Section~\ref{sectionActions} we prove Theorem~\ref{thmActions}. 

\section{Continuity of Marchenko--Ostrovskii Maps on $\bbC_+$}\label{sectionMOcontinuity}

In this section, we prove very general continuity properties of Marchenko--Ostrovskii maps away from the spectrum.

For any bounded potential $V$, there is a unique (up to normalization) nontrivial solution $\psi(x,z)$ of $(-\partial_x^2+V) \psi(x;z) = z \psi(x;z)$ which is square-integrable at $+\infty$. It is called the Weyl solution; in terms of it, the half-line Weyl functions are expressed as
\[
m(x,z) = \frac{\partial_x \psi(x;z)}{ \psi(x;z)}, \qquad z \in \bbC \setminus \bbR.
\]
We will use the following continuity property of the Weyl function:

\begin{lemma}[\cite{JMerratum}] \label{lemma:mfunctionconv}
If a sequence of potentials $V_n$ converges to $V$ uniformly on compacts, and $m_n$ are $m$-functions corresponding to $V_n$, then
\[
\lim_{n\to \infty} m_n(x;z) = m(x;z)
\]
uniformly on compacts in $z \in \mathbb{C} \setminus \mathbb{R}$.
\end{lemma}

In particular, this continuity property implies almost periodicity of $m(x, z)$ as a function of $x$ for an almost periodic potential $V$, which justifies the definition \eqref{wfunc}. We will now specialize to the quasiperiodic setting
\[
V_{U,\omega, \theta}(x) = U(\omega x+ \theta),
\]
where $U \in C(\mathbb{T}^\nu,\mathbb{R})$, $\omega \in \mathbb{R}^\nu$, $\theta \in \mathbb{T}^\nu$. We will denote the corresponding functions by $m_{U,\omega,\theta}(x;z)$ and $w_{U,\omega,\theta}(z)$. Since translation can also be written as a change of the sample point $\theta$, we have the trivial relations
\begin{align*}
m_{U,\omega,\theta}(x;z) & = m_{U,\omega,\omega x+\theta}(0;z),\\
w_{U,\omega,\theta}(z) & = w_{U,\omega,\omega x + \theta}(z).
\end{align*}

\begin{lemma}\label{lemmaWeylJointlyContinuous}
The function $m_{U,\omega, \theta}(x;z)$ is jointly continuous in $(U,\omega,\theta) \in C(\mathbb{T}^\nu,\mathbb{R}) \times \bbR^\nu \times \bbT^\nu$.
\end{lemma}

\begin{proof}
If $(U_n,\omega_n,\theta_n) \to (U,\omega,\theta)$, then the potentials $U_n(\omega_n x + \theta_n)$ converge uniformly on compacts to $U(\omega x + \theta)$, so by Lemma~\ref{lemma:mfunctionconv}, $m_{U_n,\omega_n,\theta_n}(z) \to m_{U,\omega,\theta}(z)$.
\end{proof}

In particular, for fixed $x\in \mathbb{R}$ and $z\in \mathbb{C} \setminus\mathbb{R}$, $m_{U,\omega, \theta}(x;z)$ is a continuous function of $(U,\omega,\theta)$, so it is uniformly continuous on compacts. This implies a convergence property of Marchenko--Ostrovskii maps:

\begin{prop}\label{prop:wconvoffboundary}
Let $(U_n,\omega_n,\theta_n) \to (U,\omega,\theta)$. If the components of $\omega$ are linearly independent over $\mathbb{Q}$, then for all $z \in \mathbb{C} \setminus \mathbb{R}$,
\[
\lim_{n\to\infty} w_{U_n,\omega_n,\theta_n}(z) = w_{U,\omega,\theta}(z),
\]
and convergence is uniform in compacts in $\mathbb{C} \setminus \mathbb{R}$.
\end{prop}

Note that the base space $\bbT^\nu$ is fixed in this result, but that $\omega_n$ can be rational; one of the applications is to periodic approximants of a quasiperiodic operator.

\begin{proof}
Denote $\omega = (\alpha_1, \dots, \alpha_\nu)$. Picking $t>0$ such that $1, \alpha_1 t, \dots, \alpha_\nu t$ are linearly independent over $\mathbb{Q}$ and applying Birkhoff's ergodic theorem to the discrete rotation with angle $\omega t$ implies
\begin{equation}\label{eq:womegatheta}
w_{U,\omega,\theta}(z) = \int_{\bbT^\nu} m_{U,\omega,\phi}(0;z) \, d\phi,
\end{equation}
where $d\phi$ denotes normalized Lebesgue measure on $\bbT^\nu$. Since $\omega_n$ is not assumed to have linearly independent components, this is not in general true for $w_{U_n,\omega_n, \theta_n}$. Nonetheless, we will show that
\begin{equation}\label{eq:womeganthetan}
\lim_{n\to\infty} \left\lvert w_{U_n,\omega_n, \theta_n}(z) - \int_{\mathbb{T}^\nu} m_{U_n,\omega_n,\phi}(0;z) \, d\phi \right\rvert = 0.
\end{equation}
Since $U_n \to U$ in $C(\mathbb{T}^\nu,\mathbb{R})$, this sequence is uniformly equicontinuous on $\mathbb{T}^\nu$ and $\{U_n \mid n\in\bbN \} \cup \{U\}$ is a compact subset of $C(\mathbb{T}^\nu,\mathbb{R})$. Combining the continuity from Lemma~\ref{lemmaWeylJointlyContinuous} with compactness implies uniform continuity: for any $\epsilon >0$ there is an $\delta>0$ such that $\lvert \phi - \phi' \rvert < \delta$ implies
\[
\lvert m_{U_n,\omega_n, \phi}(0;z) - m_{U_n,\omega_n, \phi'}(0;z) \rvert < \epsilon
\]
for all $n$, and therefore also
\[
\lvert w_{U_n,\omega_n, \phi}(z) - w_{U_n,\omega_n, \phi'}(z) \rvert < \epsilon.
\]
Since $\omega \mathbb{R}$ is dense in $\mathbb{T}^\nu$, for each $\epsilon > 0$ there exists $n_0$ such that for $n>n_0$, $\omega_n \mathbb{R}$ is $\epsilon$-dense in $\mathbb{T}^\nu$. Thus, for any $\phi \in \mathbb{T}^\nu$, we can find a $\phi'$ of the form $\phi' = \theta_n + \omega_n x$ such that $\lvert \phi - \phi' \rvert < \epsilon$. Since $w_{U_n,\omega_n, \theta_n} = w_{U_n,\omega_n, \theta_n + \omega_n x}$, this implies that for any $\phi$,
\[
\lvert w_{U_n,\omega_n, \theta_n}(z) - w_{U_n,\omega_n, \phi}(z) \rvert < \epsilon.
\]
Integrating in $\phi$ and using
\[
 \int_{\mathbb{T}^\nu} w_{U_n,\omega_n,\phi}(z) \, d\phi =  \int_{\mathbb{T}^\nu} m_{U_n,\omega_n,\phi}(0;z) \, d\phi,
\]
we obtain \eqref{eq:womeganthetan}.

The result now follows from \eqref{eq:womegatheta} and \eqref{eq:womeganthetan} since, by uniform continuity,
\[
\lim_{n\to\infty} \int_{\mathbb{T}^\nu} m_{U_n,\omega_n,\phi}(0;z) \, d\phi = \int_{\mathbb{T}^\nu} m_{U,\omega,\phi}(0;z) \, d\phi. \qedhere
\]
\end{proof}

Our main focus is on small quasiperiodic potentials with analytic sampling functions, and the following lemma is useful for specializing the above results.

\begin{lemma}\label{lemmaUniformCompacts}
Fix $\ve, \kappa_0 > 0$ and $\nu \in \mathbb{N}$.  Consider sampling functions $U_n, U : \bbT^\nu \to \bbR$ with exponentially decaying Fourier coefficients:
\[
\lvert \hat U_n(m) \rvert \le \ve e^{-\kappa_0 \lvert m\rvert}, \quad \lvert \hat U(m) \rvert \le \ve e^{-\kappa_0 \lvert m\rvert}, \quad \forall m \in \bbZ^\nu.
\]
Let $Q_n(x) = U_n(\omega_n x)$,   $Q(x) = U(\omega x)$, for a sequence  $\omega_n \to \omega$, where $\omega \in \bbR^\nu$ has linearly independent entries over $\mathbb{Q}$.  Then the following statements are equivalent:
\begin{enumerate}[(a)]
\item $Q_n \to Q$ uniformly on compact subsets of $\bbR$;
\item $U_n \to U$ uniformly on $\bbT^\nu$; 
\item For all $m\in \mathbb{Z}^\nu$, $\hat U_n(m) \to \hat U(m)$.
\end{enumerate}
\end{lemma}

\begin{proof}
The equivalence of (b) and (c) comes from the observation that on the set of functions of the form
\[
\sum_{m\in \bbZ^\nu} c(m) e^{2\pi i m \phi}, \qquad \lvert c(m) \rvert \le \ve e^{-\kappa_0\lvert m\rvert}
\]
pointwise convergence of Fourier coefficients is equivalent to uniform convergence on $C(\bbT^\nu)$, simply because the a priori upper bound $\ve e^{-\kappa_0\lvert m\rvert}$ is summable in $m\in\bbZ^\nu$.

(b)$\implies$(a) is trivial.

(a)$\implies$(b): We first note that
\[
\sum_{m\in\bbZ^\nu} (1 +\lvert m \rvert) e^{-\kappa_0 \lvert m \rvert} < \infty
\]
so the functions $U_n, U$ are equicontinuous. By equicontinuity,  $U_n(\omega_n x) \to U(\omega x)$ implies $U_n( \omega x) \to U(\omega x)$, and then convergence on the dense set $\{x \omega \mid x \in \bbR \} \subset \bbT^\nu$ implies uniform convergence $U_n \to U$.
\end{proof}

\section{Marchenko--Ostrovskii Maps of Periodic Approximants} \label{sectionMOperiodic}

There are many results about the size and location of the gaps of periodic spectra. Those results are typically derived for unit period, and when rescaled to arbitrary period, the estimates depend exponentially on the period.  For the purpose of periodic approximation of almost periodic operators, such estimates are only useful when the periodic approximant is superexponentially close to the almost periodic potential, as in the Pastur--Tkachenko class of limit-periodic operators.

We will instead derive estimates for the case of analytic sampling functions, which will be uniform for a sequence of periodic approximants of a quasiperiodic operator. With this motivation, even though the potential is periodic, we use a quasiperiodic representation
\begin{equation}\label{eq:FCpotentialsVtil}
\tilde V(x) = \sum_{m \in \zv} c(m) e^{2 \pi i x m  \tilde \omega}\ , \quad x \in \mathbb{R},
\end{equation}
where $\tilde \omega = (\tilde \omega_1, \dots, \tilde \omega_\nu) \neq 0$ is a vector with rational components.
The potential $\tilde V$ is $T$-periodic, where $T>0$ is uniquely determined by the requirement that
\begin{equation}\label{periodimplicit}
\{ m \tilde \omega \mid m \in \mathbb{Z}^\nu \} = \frac 1T \mathbb{Z}.
\end{equation}
To further relate the quasiperiodic representation \eqref{eq:FCpotentialsVtil} of the periodic potential $\tilde V$ to periodic theory, we introduce the following notation. By the first isomorphism theorem for groups, the map $m\mapsto m\tilde\omega$ generates an isomorphism between the additive group $T^{-1}\mathbb{Z}$ and the quotient group
\[
\mathfrak{Z}(\tilde \omega) := \mathbb{Z}^\nu/\mathfrak{N}(\tilde\omega), \quad \mathfrak{N}(\tilde\omega) := \{ m \in \zv \mid m\tilde \omega = 0 \}.
\]
We denote the coset of $m\in \zv$ by $\mathfrak{m} = m + \mathfrak{N}(\tilde\omega)$, write $\mathfrak{m}\tilde\omega = m\tilde\omega$, and equip $\mathfrak{Z}(\tilde \omega)$ with the quotient metric 
\[
\lvert \mathfrak{m} - \mathfrak{l} \rvert := \min \{ \lvert m-l\rvert : m \in \mathfrak{m}, l\in\mathfrak{l} \}.
\]
By this bijection, we relabel gaps and comb domain slits with the label $\mathfrak{m}$, which is better suited to our analysis; throughout this section, we will always assume that $n \in\bbZ$ and $\mathfrak{m} \in \mathfrak{Z}(\tilde \omega)$ are related via
\begin{equation}\label{nofm}
\frac nT = \frac{n(\mathfrak{m})}T =  \mathfrak{m} \tilde\omega.
\end{equation}
For instance, the comb domain associated with the periodic potential $\tilde V$ will be denoted in the form
\[
\Omega = \bbC_+ \setminus  \bigcup_{ \mathfrak{m} \in \mathfrak{Z}(\tilde\omega) } (\mathfrak{m} \tilde\omega, \mathfrak{m}\tilde\omega + ih_{\mathfrak{m}}] 
\]
with $h_0 = 0$ and $h_{-\mathfrak{m}} = h_\mathfrak{m}$ for all $\mathfrak{m}$.

The goal of this section is to prove two-sided estimates between exponential decay of Fourier coefficients of $\tilde V$ and exponential decay of the slits of the comb domain. The estimates will be independent of the period of the potential and dependent only on a ``Diophantine condition in a box" satisfied by the rational vector $\tilde\omega = (\ell_j / t_j)_{j=1}^\nu$, $\ell_j \in \mathbb{Z}$, $t_j \in \bbN$,
\begin{equation}\label{eq:PAI7-5-8}
|m \tilde \omega| \ge a_0 |m|^{-b_0}, \quad 0<|m|\le \ol R_0,
\end{equation}
for some
\begin{equation}\label{eq:PAIombasicTcondition}
0 < a_0 < 1,\quad \nu < b_0 < \infty,\quad (\ol R_0)^{b_0} > \prod t_j.
\end{equation}

\begin{theorem}\label{thmtildef}
There exists $\ve^\zero = \ve^\zero (a_0, b_0, \kappa_0) > 0$ such that, if $\tilde\omega$ obeys \eqref{eq:PAI7-5-8}, \eqref{eq:PAIombasicTcondition} and $T$ is given by \eqref{periodimplicit}, then for any  locally $L^2$, $T$-periodic function  $\tilde V$, the following holds:
\begin{enumerate}[(i)]
\item If $\tilde V \in \mathcal{P}(\tilde\omega,\ve,\kappa_0)$ for some $\ve < \ve^\zero$, then the heights $h_\mathfrak{m}$ obey
\[
h_\mathfrak{m} \le \ve^{1/2} \exp\left( - \frac{\kappa_0}5 \lvert \mathfrak{m} \rvert \right).
\] 
\item If, conversely, the heights obey $h_\mathfrak{m} \le \ve \exp(-\kappa |\mathfrak{m}|)$ with $0 < \ve < \ve^\zero$, $\kappa \ge 5 \kappa_0$, then, in fact, $\tilde V \in \mathcal{P}(\tilde\omega,\ve^{1/3},\frac\kappa 3)$
\end{enumerate}
\end{theorem}

The proof of Theorem~\ref{thmtildef} consists of two parts. The first part extends the method developed in \cite{DG1,DGL2,DGL1} to establish two sided estimates between subexponential decay of Fourier coefficients and of slits, under the a priori assumption of subexponential decay of Fourier coefficients. The second part uses estimates from the theory of periodic Schr\"odinger operators to remove the a priori assumption and improve the statement from subexponential decay to exponential decay. The first part is the following proposition.

\begin{prop}\label{proptwosidedhm} There exists $\ve_1 = \ve_1(a_0,b_0,\kappa_0,\nu) > 0$ such that for $\frac 12 \le \alpha_0 \le 1$, $0< \ve \le \ve_1$, if $\tilde V \in \mathcal{P}(\tilde\omega,\ve,\kappa_0,\alpha_0)$, then:
\begin{enumerate}[(1)]
\item The heights $h_\mathfrak{m}$ obey $h_\mathfrak{m} \le \ve^{1/2} \exp\left( - \frac{\kappa_0}5 \lvert \mathfrak{m} \rvert^{\alpha_0} \right)$.
\item There exists $\ve_2 = \ve_2 (a_0, b_0, \kappa_0) > 0$ such that if the $h_\mathfrak{m}$ obey 
$h_\mathfrak{m} \le \ve' \exp(-\kappa''_0 |\mathfrak{m}|^{\alpha'_0})$ with $0 < \ve' < \ve_2$, $\kappa''_0 \ge 5 \kappa_0$, $\alpha'_0 \ge \alpha_0$, then, in fact, $\tilde V \in \mathcal{P}(\tilde\omega,(\ve')^{1/3},\frac{\kappa''_0}{3},\alpha'_0)$.
\end{enumerate}
\end{prop}

It was shown in  \cite[Lemma 2.7]{DGL2} that \eqref{eq:PAI7-5-8} and \eqref{eq:PAIombasicTcondition} implies that for any $\mathfrak{m} \neq 0$,
\[
\lvert \mathfrak{m} \tilde\omega \rvert \ge a_0 \lvert \mathfrak{m} \rvert^{-b_0}.
\]
Moreover, \eqref{eq:FCpotentialsVtil} can be rewritten by grouping all terms with the same coset,
\begin{equation}\label{eq:FCpotentialsVtilcoset}
\tilde V(x) = \sum_{\mathfrak{m} \in \mathfrak{Z}(\tilde\omega)} c(\mathfrak{m} ) e^{2 \pi i x \mathfrak{m}  \tilde \omega}\ , \quad x \in \mathbb{R}.
\end{equation}

By Lemma 2.5 of \cite{DGL2}, if $\tilde V \in \mathcal{P}(\tilde\omega,\ve,\kappa_0,\alpha_0)$ for some $\alpha_0 \in [\frac 12, 1]$, then there is a constant $D= D(\kappa_0,\nu)$, non-increasing in $\kappa_0$, such that
\[
\lvert c(\mathfrak{m}) \rvert \le  D(\kappa_0,\nu) \ve \exp\left(- \frac{\kappa_0}2 \lvert \mathfrak{m}\rvert^{\alpha_0} \right), \qquad \forall \mathfrak{m} \in \mathfrak{Z}(\tilde\omega).
\]
It will sometimes be convenient to assume $D(\kappa_0,\nu)\ge 1$. The trivial, but nonetheless useful, converse is that if $\tilde V$ has a representation \eqref{eq:FCpotentialsVtilcoset} such that $\lvert c(\mathfrak{m}) \rvert \le  \ve \exp\left(- \kappa_0 \lvert \mathfrak{m}\rvert^{\alpha_0} \right)$ for all $\mathfrak{m} \in \mathfrak{Z}(\tilde\omega)$, then $\tilde V \in \mathcal{P}(\tilde\omega,\ve,\kappa_0,\alpha_0)$. 

We denote
\[
\mathfrak{Z}(\tilde\omega)_+ :=\{\mathfrak{m}\in \mathfrak{Z}(\tilde\omega): \mathfrak{m} \tilde\omega>0\}
\]
and denote, for $\mathfrak{m} \in \mathfrak{Z}(\tilde\omega)_+$, the sizes of gaps by
\[
\gamma_{\mathfrak{m}} = E^+_{\mathfrak{m}} - E^-_{\mathfrak{m}}.
\]
These are studied in \cite{DGL2}, where the following is proved:

\begin{theorem}[{\cite[Theorem $\tilde B$]{DGL2}}] \label{thmtildeb}
Assume that $\tilde V \in \mathcal{P}(\tilde\omega,\ve,\kappa_0,\alpha_0)$, where $\kappa_0 \in (0,1]$, $\alpha_0 \in [\frac 12, 1]$, and $\tilde \omega$ obeys \eqref{eq:PAI7-5-8}, \eqref{eq:PAIombasicTcondition}. There is a $\ve_0 = \ve_0(a_0,b_0,\kappa_0,\nu)>0$ such that, if $\ve \le \ve_0$, then
\begin{enumerate}[(i)]
\item The gap sizes obey
\[
\gamma_{\mathfrak{m}} \le 2  \ve \exp(-\frac{\kappa_0}{2} \lvert {\mathfrak{m}}\rvert^{\alpha_0}), \qquad \forall {\mathfrak{m}} \in \mathfrak{Z}(\tilde\omega)_+ .
\]
\item There exists $\ve^\zero = \ve^\zero (a_0, b_0, \kappa_0,\nu) > 0$ such that if $0 < \ve'' < \ve^\zero$, $\kappa'_0 \ge 4 \kappa_0$, $\alpha'_0 \ge \alpha_0$ and
\[
\gamma_{\mathfrak{m}} \le \ve'' \exp(-\kappa'_0 \lvert \mathfrak{m} \rvert^{\alpha'_0}), \qquad \forall {\mathfrak{m}} \in \mathfrak{Z}(\tilde\omega)_+ ,
\]
then, in fact, $\tilde V \in \mathcal{P}(\tilde\omega,\sqrt{2\ve''} ,\frac{\kappa'_0}{2} ,\alpha'_0)$.
\end{enumerate}
\end{theorem}

In the study of periodic Schr\"odinger operators, the basic object is the discriminant $\Delta$, which is the trace of the monodromy matrix and is related to the Marchenko--Ostrovskii map by
\begin{equation}\label{eqnDeltaMO}
\Delta(z) = 2 \cos( i T w(z)),
\end{equation}
where $T$ denotes the period. In each gap closure $[E_{\mathfrak{m}}^-, E_{\mathfrak{m}}^+]$, the derivative of the discriminant has precisely one zero $\lambda_{\mathfrak{m}}$, and it has no other zeros in $\bbC$. The derivative of the Marchenko--Ostrovskii map has the product representation
\begin{equation}\label{thetaprimeproduct}
w'(E) = - \left( \frac 1{4(E_0 -E)} \prod_{\mathfrak{m}' \in \mathfrak{Z}(\tilde \omega)_+} \frac{(\lambda_{\mathfrak{m}'}- E)^2}{(E_{\mathfrak{m}'}^- -  E)(E_{\mathfrak{m}'}^+ - E)} \right)^{1/2},
\end{equation}
which follows from the product representations of $\Delta^2- 4$ and $\Delta'$ (see, e.g., \cite[Section VII.B]{KP}) after differentiating \eqref{eqnDeltaMO}. The representation is derived on $\bbC_+$ but extends by analyticity to $\bbC \setminus \cS$. In particular, $w'$ is analytic on $\bbC \setminus \cS$ and $\lambda_{\mathfrak{m}}$ are zeros of $w'$. This product formula can also be concluded from a connection to Schwarz--Christoffel mappings, and can be written as an exponential Herglotz representation, see e.g. \cite{EY}, \cite[Section 6]{EL}.  In particular, if $\sum \gamma_m < \infty$, the exponential Herglotz representation has the form
\[
i w'(z) = \frac 1{2\sqrt{E_0 - z}} e^{\int_{[E_0,\infty) \setminus \cS} \frac{\xi(\lambda)}{\lambda - z} \,d\lambda}, 
\]
where
\[
\xi(\lambda) = \begin{cases} 1 & \exists m: \lambda \in (E_m^-, \lambda_m) \\
 -1 & \exists m: \lambda \in (\lambda_m,E_m^+) \\
0 & \text{else}
\end{cases},
\]
compare \cite{EY}, \cite[Section 6]{EL}.

For distinct $\mathfrak{m}, \mathfrak{m}' \in \mathfrak{Z}(\tilde \omega)_+$, denote
\[
\eta_{\mathfrak{m},\mathfrak{m}'} = \dist([E_\mathfrak{m}^-,E_\mathfrak{m}^+], [E_{\mathfrak{m}'}^-, E_{\mathfrak{m}'}^+]).
\]
Denote also
\[
\eta_{\mathfrak{m},0} = \dist ([E_\mathfrak{m}^-,E_\mathfrak{m}^+],  E_0)
\]
for $\mathfrak{m} \in \mathfrak{Z}(\tilde \omega)_+$, and define the constants
\begin{equation}\label{Amdefinition}
A(\mathfrak{m}) = \max \left( 2(\eta_{\mathfrak{m},0}+\gamma_\mathfrak{m})^{1/2}, \frac 1{2\eta_{\mathfrak{m},0}^{1/2}} \right) \prod_{\substack{\mathfrak{m}'  \in \mathfrak{Z}(\tilde \omega)_+ \\ \mathfrak{m}' \neq \mathfrak{m}}} \left( 1 + \frac{\gamma_{\mathfrak{m}'}}{\eta_{\mathfrak{m},\mathfrak{m}'}} \right)^{1/2}.
\end{equation}

\begin{lemma} There exists $C = C(a_0,b_0,\kappa_0,\nu)<\infty$ such that, under the assumptions of Theorem~\ref{thmtildeb}, for any $\mathfrak{m}$ with $\mathfrak{m}\omega > 0$,
\begin{equation}\label{productAmbound1}
A(\mathfrak{m}) \le C \lvert \mathfrak{m} \rvert^{C \log_2 \log_2 \lvert \mathfrak{m} \rvert}.
\end{equation}
\end{lemma}

\begin{proof}
Versions of this lemma have already appeared in \cite{BDGL1, DGL2} (compare \cite[Lemma 6.1]{BDGL1}), with the following caveats. In previous versions of this statement, the definition of $A(\mathfrak{m})$ had $1 / (2 \eta_{\mathfrak{m},0}^{1/2})$ instead of the more complicated prefactor in front of the product; since either of the prefactors is polynomially bounded in $\lvert \mathfrak{m}\rvert$, this change is irrelevant. The estimate was stated and proved on $\mathbb{Z}^\nu$, but the proofs go through without modification when  terms are indexed in $\mathfrak{Z}(\tilde\omega)$. Finally, this estimate has been proved for $\alpha_0=1$, but the same proof goes through uniformly in $\alpha_0 \in [1/2,1]$,  with the obvious modification of replacing some exponential terms by subexponential terms.
\end{proof}

The constants $A(\mathfrak{m})$ appear as bounds for infinite product representations such as the one for $w'$. Through this, they lead to two-sided estimates between gap sizes $\gamma_{\mathfrak{m}}$ and heights $h_{\mathfrak{m}}$:

\begin{lemma}\label{lemma:productloglogm}
Under the assumptions of 
Theorem~\ref{thmtildeb},  $\gamma_\mathfrak{m}$ and $h_\mathfrak{m}$ are related by two-sided estimates
\begin{equation}\label{heightgapestimate}
\frac 12 A(\mathfrak{m})^{-1}  \gamma_\mathfrak{m}  \le h_\mathfrak{m} \le  2 A(\mathfrak{m}) \gamma_\mathfrak{m}.
\end{equation}
\end{lemma}

\begin{proof}
For $\mathfrak{m}\in \mathfrak{Z}(\tilde \omega)_+$, the height $h_\mathfrak{m}$ can be expressed as
\[
h_\mathfrak{m} = \max_{z \in [E_\mathfrak{m}^-,E_\mathfrak{m}^+]}  \Re (- w(z)) = \Re (-w(\lambda_\mathfrak{m} ))
\]
since $\Re (-w)$ is increasing on $[E_\mathfrak{m}^-,\lambda_\mathfrak{m}]$ and decreasing for $[\lambda_\mathfrak{m},E_\mathfrak{m}^+]$. Since $w'$ is real-valued on $(E_\mathfrak{m}^-,E_\mathfrak{m}^+)$ and changes sign only at $\lambda_\mathfrak{m}$, it follows that
\[
2 h_\mathfrak{m} = \int_{E_\mathfrak{m}^-}^{E_\mathfrak{m}^+} \left\lvert w'(E)\right\rvert dE.
\]
For $E \in [E_\mathfrak{m}^-, E_\mathfrak{m}^+]$ and $\mathfrak{m}'\neq \mathfrak{m}$,
\[
\left( 1 + \frac{\gamma_{\mathfrak{m}'}}{\eta_{\mathfrak{m},\mathfrak{m}'}} \right)^{-1} \le \left\lvert \frac{(\lambda_{\mathfrak{m}'}-E)^2}{(E_{\mathfrak{m}'}^- - E)(E_{\mathfrak{m}'}^+ - E)} \right\rvert \le \left( 1 + \frac{\gamma_{\mathfrak{m}'}}{\eta_{\mathfrak{m},\mathfrak{m}'}} \right).
\]
Similarly, $\eta_{\mathfrak{m},0} \le \lvert  E_0 - E \rvert \le \eta_{\mathfrak{m},0} + \gamma_\mathfrak{m}$, and inserting these estimates into \eqref{thetaprimeproduct} we get precisely products of the form \eqref{Amdefinition}, obtaining
\[
A(\mathfrak{m})^{-1} \int_{E_\mathfrak{m}^-}^{E_\mathfrak{m}^+} \frac{\lvert \lambda_{\mathfrak{m}}-E \rvert }{\sqrt{ (E_{\mathfrak{m}}^- - E)(E_{\mathfrak{m}}^+ - E) }} dE \le 2 h_\mathfrak{m} \le A(\mathfrak{m}) \int_{E_\mathfrak{m}^-}^{E_\mathfrak{m}^+} \frac{\lvert \lambda_{\mathfrak{m}}-E \rvert }{\sqrt{ (E_{\mathfrak{m}}^- - E)(E_{\mathfrak{m}}^+ - E) }} dE.
\]
We rewrite the remaining integral by the change of variables $E = E_\mathfrak{m}^- + (E_\mathfrak{m}^+ - E_\mathfrak{m}^-) \sin^2(t/2)$, and obtain
\[
 \int_{E_\mathfrak{m}^-}^{E_\mathfrak{m}^+} \frac{\lvert \lambda_{\mathfrak{m}}-E \rvert }{\sqrt{ (E_{\mathfrak{m}}^- - E)(E_{\mathfrak{m}}^+ - E) }} dE = \tfrac 12 \gamma_\mathfrak{m} \int_0^{\pi} \lvert \cos t - a \rvert dt,
\]
where $a = \frac{2\lambda_\mathfrak{m} - E_\mathfrak{m}^+ - E_\mathfrak{m}^-}{E_\mathfrak{m}^+ - E_\mathfrak{m}^-}  \in (-1,1)$. The elementary estimates
$2 \le \int_0^\pi \lvert \cos t - a \rvert dt \le \pi$
 complete the proof.
\end{proof}

\begin{proof}[Proof of Proposition~\ref{proptwosidedhm}]
By \eqref{productAmbound1}, we can choose $\ve_1>0$ small enough so that for $0 < \ve < \ve_1$, we have
\[
A(\mathfrak{m}) \le \frac 1{4D \ve^{1/2}} \exp\left(\frac{\kappa_0}{20} \lvert \mathfrak{m} \rvert^{1/2} \right).
\]
(i) Combining \eqref{heightgapestimate} and Theorem~\ref{thmtildeb}(i),
\[
h_\mathfrak{m} \le 4 A(\mathfrak{m}) D \ve \exp\left(-\frac{\kappa_0}4 \lvert \mathfrak{m}\rvert^{\alpha_0} \right)  \le \ve^{1/2} \exp\left(- \frac{\kappa_0}5 \lvert \mathfrak{m}\rvert^{\alpha_0} \right).
\]
(ii) Using \eqref{heightgapestimate}, we obtain
\[
\gamma_\mathfrak{m} \le 2 A(\mathfrak{m}) \ve' \exp\left(- \kappa_0''  \lvert \mathfrak{m}\rvert^{\alpha_0'} \right)
 \le \ve'' \exp\left(-\kappa'_0 \lvert \mathfrak{m}\rvert^{\alpha'_0} \right)
\]
where $\kappa'_0 = \kappa''_0 - \frac{\kappa_0}{20} \ge 4 \kappa_0$ and $\ve'' = \tfrac 12 \ve_1^{-1/2}\ve'$. Choosing $\ve_2 = \min (\ve_1^{3/2}, 2 \ve_1^{1/2} \ve^{(0)})$, we ensure that $\ve'' < \ve^{(0)}$, so we can apply Theorem~\ref{thmtildeb}(ii) to obtain $\tilde V \in \mathcal{P}(\tilde\omega, \sqrt{2\ve''},\frac{\kappa'_0}2,\alpha_0')$, which implies $\tilde V \in \mathcal{P}(\tilde\omega, (\ve')^{1/3}, \frac{\kappa''_0}3, \alpha'_0)$.
\end{proof}

The second part of the proof of Theorem~\ref{thmtildef} consists of a continuity argument used to remove the a priori subexponential decay assumption from Proposition~\ref{proptwosidedhm}(2). We need to invoke the following fundamental results from the inverse spectral theory of periodic Schr\"odinger operators; see \cite{MO,PT,Tr,GT,GT2}. Let $Q$ be a real $T$-periodic potential, $Q \in L^2([0,T])$,  given by
\begin{equation} \label{eq:8potential}
Q(x) = \sum_{n \in \mathbb{Z}} d(n) e^{\frac{2\pi i nx}{T}}.
\end{equation}
Consider the spectrum $\cS = \sigma(H_Q)$. The isospectral torus $\mathcal{R}(\cS)$ consists of $T$-periodic potentials $q$ such that $\sigma(H_q) = \cS$. We recall the following facts:
\begin{enumerate}[(a)]
\item The gaps $(E^-_n,E^+_n)$ in the spectrum obey
\begin{equation} \label{eq:Hillgapsexpdecayl2}
\sum_{n=1}^\infty (E_n^+ - E_n^-)^2 <+\infty
\end{equation}
and the heights $h_n$ obey
\[
\sum_{n=1}^\infty n^2 h_n^2 < +\infty.
\]

\item Given a sequence $h_n\ge 0$, $n=1,\dots$ with $\sum_n n^2 h_n^2<+\infty$, there exists a unique up to a shift sequence of intervals $(E^-_n,E^+_n)$ such that these intervals are the gaps in the spectrum for some $q\in L^2(0,T)$ and the $h_n$'s are the heights of the slits in the associated comb domain.

\item Let $\mu_n(q)$ be the Dirichlet eigenvalues and $\nu_n(q)$ be the Neumann eigenvalues for $q$. Then
\[
E^-_n\le \mu_n(q),\nu_n(q)\le E^+_n.
\]
Furthermore, $q$ is even if and only if
\[
\mu_n(q)\in \{E^-_n,E^+_n\},\quad \forall n \in \bbN,
\]
and if and only if
\[
\nu_n(q)\in \{E^-_n,E^+_n\},\quad \forall n \in \bbN.
\]

\item Denote by $\mathcal{E}_0$ the set of all real, $T$-periodic even $L^2$ functions with zero average. Given $q\in \mathcal{E}_0$
consider the sequence $(\mu_n(q)-\nu_n(q))_{n=1}^\infty$. Since $\mu_n(q)-\nu_n(q)=\pm (E^+_n-E^-_n)$, the values $\mu_n(q)-\nu_n(q)$ are called
the signed gap-lengths. Similarly, the quantities
\[
\rho_n(q) = \sgn(\mu_n(q) - \nu_n(q)) h_n(q)
\]
are called the signed heights. The map $\Upsilon: q\to (\rho_n(q))_{n=1}^\infty \in \ell_1^2$ is a real analytic diffeomorphism from $\mathcal{E}_0$ onto $\ell_1^2$.

\item Assume $E_n^- < E_n^+$ and denote $G_{n,\sigma} = \{ (\lambda,\sigma) \mid E_n^- <  \lambda < E_n^+ \}$ for $\sigma = \pm$ and 
\begin{equation}\label{eq:8toruspm-2}
\mathcal{C}_n = G_{n,-} \cup G_{n,+} \cup \{E_n^-,E_n^+\}.
\end{equation}
The set $\mathcal{C}_n$ has a natural smooth structure which makes it diffeomorphic to the circle $\mathbb{T} = \{ e^{i \theta} : 0 \le \theta \le 2 \pi \}$. For each $q$ one can define $\sigma_n(q)\in \{+,-\}$ so that the map $q\to ((\mu_n(q),\sigma_n(q)))\in \prod_n \mathcal{C}_n$
is a diffeomorphism.

\item By the inequality $\gamma_n \le \max(2\pi n h_n, h_n^2)$, if
\begin{equation}\label{hnexpdecay}
h_n \le A_1 \exp\left( -\eta_1 \frac{\lvert n \rvert}T \right),
\end{equation}
then
\[
\gamma_n \le A(T,A_1,\eta_1) \exp\left( -\eta(T,A_1,\eta_1) \frac{\lvert n \rvert}T \right).
\]

\item Assume that for all $n$,
\begin{equation} \label{eq:analyticHillgapsexpdecay}
\gamma_n \le A\exp \left( -\eta \frac{|n|}{T} \right).
\end{equation}
Then
\begin{equation} \label{eq:8coeffexpdecayQ}
|d(n)| \le C(T,A,\eta) \exp \left( -c(T,A,\eta) \frac{|n|}{T} \right).
\end{equation}
\end{enumerate}

Besides the quotient metric introduced before, $\mathfrak{Z} (\tilde \omega)$ can also be equipped with the metric pulled back from $\mathbb{R}$,
\[
\lvert \mathfrak{m} - \mathfrak{l} \rvert_{\mathbb{R}} := \min \{ \lvert m\tilde \omega - l\tilde\omega \rvert : m \in \mathfrak{m}, l\in\mathfrak{l} \}.
\]
As it has been observed in \cite{DGL1}, there are constants $c_{\tilde\omega}$, $C_{\tilde\omega}$ such that
\begin{equation}\label{equivalenceofmetrics}
c_{\tilde \omega} \lvert\mathfrak{m}\rvert_\mathbb{R} \le \lvert\mathfrak{m}\rvert \le C_{\tilde \omega} \lvert\mathfrak{m}\rvert_\mathbb{R}.
\end{equation}
Combining (f) and (g) with \eqref{equivalenceofmetrics}, we conclude that \eqref{hnexpdecay} implies
\begin{equation} \label{eq:8coeffexpdecayQ1}
|d(n(\mathfrak{m}))| \le C(T,A,\eta) \exp \left( -c_1(T,A,\eta) |\mathfrak{m}| \right).
\end{equation}

\begin{proof}[Proof of Theorem~\ref{thmtildef}] (i) follows from Prop.~\ref{proptwosidedhm} with $\alpha_0=1$.

(ii) First we prove the statement for even potentials.
Let $D:\ell_1^2(\mathfrak{Z}(\tilde\omega)_+)\to \mathcal{E}_0$ be the inverse of $\Upsilon$.
Given $\rho=(\rho_n)\in\ell_1^2(\mathbb{Z}_+)$,  denote by $d_\mathfrak{m}(\rho)$ the Fourier coefficients of $D(\rho)$.
Let
\[
\mathcal{S}(\ve,\kappa)=\{\rho=(\rho_\mathfrak{m})_{\mathfrak{m}\in\mathfrak{Z}(\tilde\omega)_+}:|\rho_\mathfrak{m}|\le \ve\exp(-\kappa |\mathfrak{m}|)\}.
\]
Given $1/2\le \alpha<1$ set
\[
\mathcal{G}(\ve,\kappa,\alpha)=\{\rho\in \mathcal{S}(\ve,\kappa): |d_\mathfrak{m}(\rho)|\le \ve^{1/3} \exp(-\tfrac{\kappa}{3} |m|^{\alpha})\}.
\]
The set $\mathcal{G}(\ve,\kappa,\alpha)$ is nonempty, since $D(0)=0$, and it is obviously closed since each $d_\mathfrak{m}$ is continuous. We claim that the set $\mathcal{G}(\ve,\kappa,\alpha)$ is also relatively open in $\mathcal{S}(\ve,\kappa)$ in the $\ell_1^2$ metric.
Indeed, note first of all that due to $(f)$ we have
\begin{equation} \label{eq:8coeffexpdecayQ2}
|d_\mathfrak{m}| \le C(T,\ve,\kappa) \exp \left( -c(T,\ve,\kappa) |\mathfrak{m}| \right)
\end{equation}
for any $\rho\in \mathcal{S}(\ve,\kappa)$. In particular, one can find $M=M(T,\ve,\kappa,\alpha)$ large enough so that for
$|\mathfrak{m}| > M$, we have
\begin{equation} \label{eq:8coeffexpdecayQ13}
|d_\mathfrak{m}| < \ve^{1/3} \exp(-\tfrac{\kappa}{3} |\mathfrak{m}|^{\alpha}).
\end{equation}
Next, note that since the map $D$ is continuous
from $\ell^2(\mathfrak{Z}(\tilde\omega)_+)$ to $L^2(0,T)$, the map $\rho\to (d_\mathfrak{m}(\rho))_{\mathfrak{m}\in \mathfrak{Z}(\tilde\omega)}$
is also continuous from $\ell^2(\mathfrak{Z}(\tilde\omega)_+)$ to $\ell^2(\mathfrak{Z}(\tilde\omega))$. Let $\rho^\zero=(\rho^\zero_\mathfrak{m})\in \mathcal{G}(\ve,\kappa,\alpha)$. Find $\delta>0$ such that 
\begin{equation} \label{eq:8coeffexpdecayQ15}
\|D(\rho)-D(\rho_0)\|<\ve^{1/3} \exp(-\kappa M^\alpha)
\end{equation}
for any $\|\rho-\rho_0\|<\delta$. Let $\|\rho-\rho_0\|<\delta$. Then, in particular, we have
\begin{equation} \label{eq:8coeffexpdecayQ16}
|d_\mathfrak{m}(\rho)| < \ve^{1/3} \exp(-\tfrac{\kappa}{3} |\mathfrak{m}|^{\alpha})+\ve^{1/3} \exp(-\kappa M^\alpha)<2\ve^{1/3} \exp(-\tfrac{\kappa}{3} |\mathfrak{m}|^{\alpha})
\end{equation}
for any $|\mathfrak{m}|\le M$. Consider $q=D(\rho)$, i.e.
\begin{equation}\label{eq:1potentialsVtilDmap}
q(x) = \sum_{m \in \mathfrak{Z}(\tilde\omega)} d_\mathfrak{m}(\rho) e^{2 \pi i x \mathfrak{m} \tilde \omega}.
\end{equation}
Proposition~\ref{proptwosidedhm} applies to $q$ in the role of $\tilde V$, as long as $2\ve^{1/3} < \ve_1(a_0,b_0,\kappa_0)$ and $\ve < \ve_2(a_0,b_0,\kappa_0)$. Thus,
\begin{equation} \label{eq:8coeffexpdecayQ17}
|d_\mathfrak{m}(\rho)| < \ve^{1/3} \exp(-\tfrac{\kappa}{3} |\mathfrak{m}|^{\alpha}),
\end{equation}
i.e. $\rho\in \mathcal{G}(\ve,\kappa,\alpha)$. 

By connectedness, it follows that $\mathcal{G}(\ve,\kappa,\alpha)=\mathcal{S}(\ve,\kappa)$, i.e. for any $\rho
\in \mathcal{S}(\ve,\kappa)$ one has
\begin{equation} \label{eq:8coeffexpdecayQ18}
|d_\mathfrak{m}(\rho)| < \ve^{1/3} \exp(-\tfrac{\kappa}{3} |\mathfrak{m}|^{\alpha}),
\end{equation}
That completes the proof for even potentials.  To prove the statement for arbitrary $\tilde V$, we first recall that there exists an even potential $q_0$ isospectral with $\tilde V$. Since we have proved the statement for even potentials, we have
\begin{equation}\label{eq:1potDmap1}
q_0(x) = \sum_{m \in \mathfrak{Z}(\tilde\omega)} d^\zero_\mathfrak{m} e^{2 \pi i x \mathfrak{m} \tilde \omega},
\end{equation}
\begin{equation} \label{eq:8coeffexpdecayQ20}
|d^\zero_\mathfrak{m}| < \ve^{1/3} \exp \left(-\tfrac{\kappa}{3} |\mathfrak{m}|^{\alpha} \right).
\end{equation}
Now we employ the map $q\to ((\mu_n(q),\sigma_n(q)))\in \prod_n \mathcal{C}_n$ from $(e)$. By another open set-closed set argument as above, it follows that the statement holds for any potential $q$ isospectral with $q_0$.

We have therefore shown that
\[
\lvert c(\mathfrak{m}) \rvert \le \ve^{1/3} \exp\left( - \tfrac \kappa 3 \lvert \mathfrak{m}\rvert^\alpha \right)
\]
for any $\alpha \in [1/2,1)$. Taking the limit $\alpha \to 1$, we conclude
\begin{equation}\label{eq:FCest2}
|c(\mathfrak{m})| \le \ve^{1/3} \exp\left(-\tfrac{\kappa}{3} |\mathfrak{m}| \right)
\end{equation}
which completes the proof.
\end{proof}

\section{Stability of Marchenko--Ostrovskii Maps} \label{sectionGeometric}

In this section, we use geometric function theory to study Marchenko--Ostrovskii maps up to the boundary. We rely on the fact that the domains for the periodic approximants satisfy some uniform estimates, which will transfer also to the limit.

We begin by recalling some general facts \cite[Chapter 4]{Po}, which we rewrite with respect to the domain $\bbC_+$ and the boundary point $\infty$ instead of $\bbD$; these rewritings are trivially obtained by pre- and post-composing with the Cayley transform as needed. Let $g:\bbC_+ \to \bbC_+$ be analytic and $\Gamma$ a Jordan arc in $\bbC_+ \cup \{\infty\}$ with endpoint $\infty$. If 
\[
\lim_{\substack{z \to \infty \\ z \in \Gamma}} g(z) = \infty,
\]
then $g$ has the angular limit $\infty$ at $\infty$, i.e., for any $\alpha > 0$,
\begin{equation}\label{eqnAngularLimit}
\lim_{\substack{z \to \infty \\  \arg z \in (\alpha,\pi - \alpha)}} g(z) = \infty.
\end{equation}
By the Julia--Wolff lemma \cite[Prop. 4.13]{Po},  \eqref{eqnAngularLimit} implies that $g$ has an angular derivative at $\infty$, in the sense that for every $\alpha > 0$, the limit
\[
g'(\infty) := \lim_{\substack{z \to \infty \\  \arg z \in (\alpha,\pi - \alpha)}} \frac {z}{g(z)}
\]
exists and $g'(\infty) \in (0,\infty]$.

It is convenient to transform a Marchenko--Ostrovskii map into the form
\begin{equation}\label{eqnMOtransform}
\Theta(z) = -i w(  E_0 + z^2).
\end{equation}
The map $\Theta$ is a conformal bijection from $\bbC_+$ to the Marchenko--Ostrovskii domain $\Omega$, and the asymptotics \eqref{eqn2term} implies that $\Theta$ has a finite angular derivative at $\infty$, namely,
\[
\Theta'(\infty) =1.
\]
In geometric function theory it is more convenient to use an internal normalization, and our work  will switch between those two. 
The following corollary of the Schwarz--Pick theorem is useful for comparing angular derivatives at $\infty$.

\begin{lemma}\label{comparisonSP}
Let $g_1, g_2: \bbC_+ \to \bbC_+$ be analytic, injective, and have the symmetry
\[
g_j(z) = - \ol{ g_j( - \ol z)}.
\]
If $g_1, g_2$ satisfy \eqref{eqnAngularLimit} and $\{iy \mid 1 \le y < \infty\} \subset g_1(\bbC_+) \subset g_2(\bbC_+)$, then
\begin{equation}\label{eqschwarzpick}
\frac{1} {g_2'(\infty)} \ge \frac{g_1^{-1}(i)}{g_2^{-1}(i)} \frac{1} {g_1'(\infty)}.
\end{equation}
\end{lemma}

\begin{proof}
By the Schwarz--Pick theorem, analytic maps $g:\mathbb{C}_+ \to \mathbb{C}_+$ don't increase distances in the Poincar\'e metric, i.e.
\begin{equation}\label{SchwarzPick}
\left\lvert \frac{g(z_1) - g(z_2)}{\overline{g(z_1)} - g(z_2)}  \right\rvert \le \left\lvert \frac{z_1 - z_2}{\overline{z_1} - z_2}  \right\rvert, 
\quad z_1, z_2 \in \mathbb{C}_+.
\end{equation}
Applying \eqref{SchwarzPick} to $g = g_2^{-1} \circ g_1: \mathbb{C}_+ \to \mathbb{C}_+$ and taking $z_j = g_1^{-1}(w_j)$ implies
\[
\left\lvert \frac{g_2^{-1}(w_1) - g_2^{-1}(w_2)}{\overline{g_2^{-1}(w_1)} - g_2^{-1}(w_2)}  \right\rvert \le \left\lvert \frac{g_1^{-1}(w_1) - g_1^{-1}(w_2)}{\overline{g_1^{-1}(w_1)} - g_1^{-1}(w_2)}  \right\rvert, \quad w_1, w_2 \in g_1(\bbC_+).
\]
Dividing by $w_1 - w_2$ and letting $w_1 \to w_2 =iy$, we obtain
\begin{equation}\label{SchwarzPick2}
\frac{\lvert (g_2^{-1})'(iy) \rvert}{\Im (g_2^{-1})(iy)} \le \frac{\lvert (g_1^{-1})'(iy) \rvert}{\Im (g_1^{-1})(iy)}, \; y \ge 1.
\end{equation}
Since $g_j^{-1}$ are purely imaginary along the imaginary axis, integrating along the segment $[i,iy]$ and letting $y \to +\infty$ gives \eqref{eqschwarzpick}.
\end{proof}

\begin{lemma}\label{lemmaconformalmaps}
Consider a simply connected domain $\Omega \subset \mathbb{C}_+$ symmetric under reflection with respect to the imaginary axis, $z\leftrightarrow - \bar z$. Assume that there exists $\epsilon \in (0,1)$ such that $\{ z\in \mathbb{C} \mid \Im z > \epsilon\} \subset \Omega$. Let $f : \mathbb{C}_+ \to \Omega$ be the conformal bijection satisfying the normalization
\begin{equation}\label{falphanormalization}
f(i)=i, \quad f'(i) > 0.
\end{equation}
Then,
\begin{enumerate}[(i)]
\item The map $f$ has the symmetry
\begin{equation}\label{falphasymmetry}
f(z) =  -  \overline{ f(-\bar z)}.
\end{equation}
\item The map $iy \mapsto f(iy)$ is a strictly increasing bijection $(0,i\infty) \to (0,i\infty)\cap \Omega$.
\item The map $f$ has an angular limit at $\infty$: for all $\alpha > 0$,
\begin{equation}\label{falphalimits}
\lim_{\substack{z \to \infty \\ \arg z \in (\alpha,\pi - \alpha)}} f(z) = \infty.
\end{equation}
\item The angular derivative $f'(\infty)$  exists and
\begin{equation}\label{scalingconstantbound} 
1 \le f'(\infty) \le \frac 1{1-\epsilon}.
\end{equation}
\item The map $\Theta$ given by
\begin{equation}\label{scalingMO}
\Theta (z) = f(f'(\infty) z)
\end{equation}
is the unique conformal bijection from $\mathbb{C}_+$ onto $\Omega$ obeying the normalization
\begin{equation}\label{normalizationconformal}
\Theta(z) = - \overline{\Theta(-\bar z)}, \quad \lim_{y\to\infty} \frac{\Theta(iy)}{iy} = 1.
\end{equation}
\end{enumerate}
\end{lemma}

\begin{proof}
(i) Since $\Omega$ is symmetric with respect to the imaginary axis, $-\overline{f(-\bar z)}$ is also a conformal map of $\mathbb{C}_+$ to $\Omega$, and it satisfies the same normalization \eqref{falphanormalization}, so they are equal.

(ii) By the symmetry \eqref{falphasymmetry}, if $f(z) = it$ for some $t\in \mathbb{R}$, then also $f(-\bar z) = it$; since $f$ is injective, this implies $z = -\bar z$. Conversely, $z=-\bar z$ implies $f(z) \in (0,i\infty)$ by the same symmetry  \eqref{falphasymmetry}. Thus, $f$ maps $(0,i\infty)$ to $(0,i\infty) \cap \Omega$ bijectively. By continuity, $y\mapsto -if(iy)$ is strictly increasing or strictly decreasing; since $f'(i)>0$, it is strictly increasing. Thus, 
 $\lim_{y\to \infty} f(iy) = \infty$. 

(iii) It follows from (ii) that $f(z)$ has a radial limit at  $\infty$ along the imaginary axis. Thus, $f$ has an angular limit at $\infty$. 

(iv) The angular derivative $f'(\infty)$ exists by the Julia--Wolff lemma. Applying Lemma~\ref{comparisonSP} to functions $f$ and $z \mapsto z$, we obtain $f'(\infty) \ge 1$. Similarly, applying this to $z \mapsto z+ i \epsilon$ and $f$, we obtain $f'(\infty) \le 1 / (1-\epsilon)$.

(v) For any such conformal map $\Theta$, symmetry dictates that there is a unique $C>0$ such that $\Theta(i/C) = i$ and that $\Theta'(i/C) >0$. Thus, $\Theta(z/C) =f(z)$. Comparing derivatives at $\infty$, we see $\Theta'(\infty) = C^{-1} f'(\infty)$, so $C= f'(\infty)$.
\end{proof}

We now specialize to a family of comb domains which satisfies some uniform decay conditions on the slit heights and uniform upper bounds and Diophantine conditions for the frequency vectors. 

\begin{lemma}\label{propequicontinuity} Let $\{\Omega_\alpha\}_{\alpha \in \cA}$ be a family of comb domains of the form
\[
\Omega_\alpha = \mathbb{C}_+ \setminus \bigcup_{\mathfrak{m} \in\mathfrak{Z}(\omega_\alpha)}  
( \mathfrak{m}\omega_\alpha,  \mathfrak{m}\omega_\alpha + i h_{ \mathfrak{m}, \alpha} ]
\]
where $\omega_\alpha \in \mathbb{R}^\nu$ and $h_{\mathfrak{m},\alpha} \ge 0$ for $\mathfrak{m} \in \mathfrak{Z}(\omega_\alpha)$. Assume that the following hold, for some constants $r< \infty$, $a_0 \in (0,1)$, $b_0 \in (\nu,\infty)$, $\ve > 0$, $\kappa>0$ independent of $\alpha \in \cA$:
\begin{enumerate}[(i)]
\item $\lvert \omega_\alpha \rvert < r$ and $\omega_\alpha$ obey the Diophantine condition
\begin{equation}\label{Diophantineomega}
\lvert \mathfrak{m} \omega_\alpha \rvert \ge a_0 \lvert \mathfrak{m} \rvert^{-b_0}, \quad \forall \mathfrak{m}\in \mathfrak{Z}(\omega_\alpha) \setminus\{0\}.
\end{equation}
\item $h_{0,\alpha}=0$, $h_{\mathfrak{m},\alpha} = h_{-\mathfrak{m},\alpha} \ge 0$, and
\begin{equation}\label{hmexpdecay}
 0 \le h_{\mathfrak{m},\alpha} \le \ve e^{-\kappa \lvert \mathfrak{m} \rvert}, \qquad \forall \mathfrak{m} \in \mathfrak{Z}(\omega_\alpha).
\end{equation}
\end{enumerate}
Then, for each $\alpha \in \cA$, there is a unique conformal map $\Theta_\alpha:\mathbb{C}_+ \to \Omega_\alpha$ with the normalization \eqref{normalizationconformal}, 
and the family of functions $\{\Theta_\alpha\}_{\alpha\in \cA}$ is equicontinuous on $\ol{ \bbC_+}$.
\end{lemma}

We use the spherical metric on $\hat\bbC$, given by
\[
d(z_1,z_2) = \frac{\lvert z_1 - z_2 \rvert}{\sqrt{1+\lvert z_1 \rvert^2} \sqrt{1+\lvert z_2 \rvert^2}}, \qquad z_1, z_2 \in \mathbb{C}
\]
and $d(z_1,\infty) = d(\infty,z_1) = \frac 1{\sqrt{1+\lvert z_1\rvert^2}}$. 

\begin{proof}
In geometric function theory, sets $\hat\bbC \setminus \Omega_\alpha$ are said to be uniformly locally connected if  for every $\epsilon > 0$, there exists $\delta > 0$ such that for all $\alpha$ and all $z,w \in \hat\bbC \setminus \Omega_\alpha$ with $d(z,w) < \delta$, there exists a connected set $B \subset \hat\bbC \setminus \Omega_\alpha$ such that $z,w \in B$ and $\lvert B \rvert < \epsilon$. We use $\lvert \cdot \rvert$ to denote the diameter of a set.

Consider two boundary points $z = x+ iy$, $z' = x' + i y'$. Let $c$ be the shorter of the two geodesics connecting $z, z'$, parametrized by $[0,1]$ so that $c(0) = z$ and $c(1)=z'$.

Let us first take care of the simple case $x=x'$ and without loss of generality let $y \le y'$. In that case, we choose $B = [z,z']$ and see that $\lvert B \rvert = d(z,z')$. 

From now on assume $x \neq x'$. Then, $B$  can be chosen to be the union of three segments
\[
B = [x, x+iy] \cup I \cup  [x', x'+ iy'],
\]
where $I \subset \overline{\mathbb{R}}$ is the shorter of the two geodesics in $\hat\bbC$ connecting $x$ and $x'$. Since $y, y' \in [0,\ve]$, it follows that
\[
\lvert c \rvert \ge d(x+iy, x'+iy') \ge \frac 1{1+\ve^2} d(x, x') = \frac 1{1+\ve^2} \lvert I \rvert.
\]
Also,
\[
\left\lvert \{ x + i t : t \in [0, y]\} \right\rvert  = d(x, x+iy) \le y.
\]
Using $c$ and two of the three line segments to estimate the diameter of $c \cup B$, we see that
\[
\lvert  B \rvert = \lvert c \cup B \rvert \le (2+\ve^2) \lvert c\rvert + Y, \qquad Y=\min(y,y').
\]
The only case when $Y \neq 0$ is when $c(0), c(1)$ lie on slits, i.e.
\[
c(0) = \mathfrak{m}\omega_\alpha + iy, \quad c(1) = \mathfrak{m}'\omega_\alpha + i y'.
\]
From exponential decay of the $h_\mathfrak{m}$, we conclude
\[
\lvert \mathfrak{m}\rvert , \lvert \mathfrak{m}'\rvert \le \frac 1{\kappa} \ln \frac{\ve}Y.
\]
From the Diophantine condition we see
\[
\lvert (\mathfrak{m} - \mathfrak{m}') \omega_\alpha \rvert \ge a_0 \lvert \mathfrak{m}-\mathfrak{m}' \rvert^{-b_0} \ge a_0 \left(\frac 2{\kappa} \ln \frac{\ve}Y \right)^{-b_0},
\]
so
\[
\lvert c \rvert \ge \frac 1{1+\ve^2} d(\mathfrak{m}\omega_\alpha, \mathfrak{m}' \omega_\alpha)  \ge \frac 1{1+\ve^2}  \frac{\lvert (\mathfrak{m}-\mathfrak{m}')\omega_\alpha\rvert}{\sqrt{1+\lvert \mathfrak{m}\omega_\alpha \rvert^2}\sqrt{1+\lvert \mathfrak{m}'\omega_\alpha \rvert^2}} \ge \frac 1{1+\ve^2}  \frac{ a_0\left(\frac 2{\kappa} \ln \frac{\ve}Y \right)^{-b_0} }{1 + \left( \frac 1{\kappa} \ln \frac{\ve}Y r \right)^2} =: F(Y).
\]
Note that $F(Y)$ is a strictly increasing function of $Y$ depending only on $a_0,b_0,r,\ve,\kappa$, and $\lim_{Y\to 0} F(Y) = 0$. Combining this with previous estimates, we see that
\[
\lvert B \rvert \le (2+\ve^2) \lvert c \rvert + F^{-1}(\lvert c \rvert),
\]
where $F^{-1}$ is the inverse function of $F$ on $(0,\infty)$.

Thus, the domains $\Omega_\alpha$ are uniformly locally connected (see \cite[Chapter 2]{Po}; this reference is formulated for bounded domains, with Euclidean distance, but by conjugating our maps by the Cayley transform, it reduces to our setting of domains in $\bbC_+$). In particular, by a uniform version of Carath\'eodory's theorem \cite[Prop. 2.3]{Po}, the conformal bijections $f_\alpha: \bbC_+ \to \Omega_\alpha$ with
\[
f_\alpha(i)=i, \qquad f_\alpha'(i) > 0
\]
have continuous extensions to $\ol{\bbC_+}$ and are uniformly equicontinuous in $\ol{\bbC_+}$. Thus, the maps $\Theta_\alpha$ are uniquely determined by
\[
\Theta_\alpha(z) = f_\alpha( f_\alpha'(\infty) z).
\]
Since $1 \le f_\alpha'(\infty) \le (1-\ve)^{-1}$ and $f_\alpha$ are uniformly equicontinuous, $\Theta_\alpha$ are uniformly equicontinuous.
\end{proof}

The above setup includes both the case of irrational, Diophantine frequency vectors and the case of rational frequency vectors obeying a Diophantine condition in a box \eqref{eq:PAI7-5-8}, \eqref{eq:PAIombasicTcondition}. Indeed, the first application of this equicontinuity statement will be  to a sequence of rational approximants of the Diophantine frequency $\omega$; the second application will be to the constant sequence $\omega_n = \omega$.

To geometrically characterize convergence of conformal maps, Carath\'eodory introduced the following notion of kernel convergence \cite[Section 1.4]{Po}, with respect to an internal point which we take to be $i$. Let $\Omega_n \subset \bbC$ be domains containing $i$. It is said that $\Omega_n \to \Omega$ in the sense of kernel convergence with respect to $i$ if:
\begin{enumerate}[(i)]
\item either $\Omega= \{i\}$, or $\Omega \subsetneq \bbC$ is a domain with $i\in \Omega$ such that every $w\in \Omega$ has a neighborhood which lies in $\Omega_n$ for all sufficiently large $n$
\item for every $w\in \partial\Omega$ there exist $w_n \in \partial\Omega_n$ such that $w_n\to w$ as $n\to\infty$
\end{enumerate}

\begin{prop}\label{propMOconvergence}
Take a sequence $\{\Omega_n\}_{n=1}^\infty$ from the family of comb domains $\{\Omega_\alpha\}_{\alpha\in \cA}$ considered in Lemma~\ref{propequicontinuity}. Assume that the limit
\begin{equation}\label{convergenceomega}
\lim_{n\to\infty} \omega_{n} = \omega
\end{equation}
exists and that $\omega$ has linearly independent entries over $\bbQ$, i.e., $m\omega \neq 0$ for all $m\in\zv \setminus\{0\}$. 
Then the following are equivalent:
\begin{enumerate}[(i)]
\item For every $m\in\zv$, the limit
\begin{equation}\label{hmlimit}
h_{m} := \lim_{n\to\infty} h_{\mathfrak{m}, n}
\end{equation}
exists (note that we use $\mathfrak{m} = m+\mathfrak{N}(\omega_{n})$ to denote the coset containing $m\in\zv$ in an $n$-dependent Abelian group).
\item The maps $\Theta_n$ have a pointwise limit on $\{z \in \bbC_+ \mid \Re z > 0\}$.
\item The maps $\Theta_n$ converge uniformly on $\ol{\bbC_+}$.
\end{enumerate}
In this case, the limit of $\Theta_n$ is the conformal map $\Theta$ corresponding to the comb domain
\begin{equation}\label{Omegacombdomain}
\Omega = \mathbb{C}_+ \setminus \bigcup_{{m} \in\zv} (m\omega, m\omega+ih_m]. 
\end{equation}
\end{prop}

\begin{proof}
(i)$\implies$(ii): We begin by observing that, since $\mathfrak{N}(\omega) = \{0\}$,  for each fixed $m\in \zv\setminus\{0\}$, $\lim_{n\to\infty} m\omega_n = m\omega \neq 0$, so $m\omega_n \neq 0$ for large enough $n$. Thus, for all $M\in\mathbb{N}$ there exists $N=N(M)\in\mathbb{N}$ such that
\begin{equation}\label{vanishingkernel}
\{ m\in \zv \mid \lvert m \rvert \le M \} \cap \mathfrak{N}(\omega_n) = \{ 0\}, \qquad \forall n\ge N.
\end{equation}
In particular, \eqref{vanishingkernel} implies  for any $m\in \zv$, we have $\lvert m\rvert = \lvert m + \mathfrak{N}(\omega_n) \rvert$ for all large enough $n$, so the uniform estimates
\eqref{Diophantineomega}, \eqref{hmexpdecay} extend to the limit, and we can consider the comb domain $\Omega$ to be part of our family $\{\Omega_\alpha\}_{\alpha \in \cA}$.

Assume $z \in \Omega$ and fix $L < \Im z$. There exists $m_0 \in \bbN$ such that $\lvert m \rvert >  m_0$ implies $h_{\mathfrak{m},n} < L$. Of the remaining finitely many slits with $\lvert m \rvert \le m_0$, each has $h_{m,0} < \Im z$ or $m\omega_0 \neq \Re z$, so $\liminf \dist(z,\Omega_n) > 0$.

Now assume $z \in \partial\Omega$. If $z = m \omega + iy$, the points $z_n = m \omega_n + i \min\{y, h_{m,n}\} \in \partial\Omega_n$ converge to $z$; if $z \in \ol\bbR$, we may take $z_n = z\in \partial\Omega_n$.

Thus, pointwise convergence \eqref{hmlimit} implies kernel convergence of domains $\Omega_n$ to $\Omega$ with respect to the point $i$ (see \cite[Section 1.4]{Po}), so it implies that $f_{n} \to f$ uniformly on $\ol{\bbC_+}$, where the $f_n$'s are the conformal bijections to $\Omega_n$ satisfying the normalization \eqref{falphanormalization}. Since $\Theta_n(z) = f_n(f_n'(\infty) z)$, it remains to prove that $f_n'(\infty) \to f'(\infty)$.

Fix $\delta> 0$ and let $M = M(\delta) \in\mathbb{N}$ be chosen so that $\ve e^{-\kappa M} < \delta$. Denote by $g_{\alpha,M}$ the conformal maps $\mathbb{C}_+ \to \mathbb{C}_+ \setminus \bigcup_{\lvert \mathfrak{m} \rvert \le M} \{ \mathfrak{m} \omega_\alpha + i [0,h_{\mathfrak{m},\alpha}] \}$ with $g_{\alpha,M}(i)=i$, $g'_{\alpha,M}(i) > 0$.  By the same argument as above, the conformal bijections $g_{n,M}$ converge to $g_{M}$ uniformly on $\overline{\mathbb{C}_+}$. Moreover, there is a finite supremum
\[
L = L(M) = \sup_n g_{n,M}( M \lvert \omega_{n} \rvert ) < \infty
\]
such that $g_{n,M}(x) \in \mathbb{R}$ if $x\in\mathbb{R}$ with $\lvert x \rvert > L$. Using the reflection principle on the functions $1/g_{n,M}(1/z)$ in a neighborhood of $0$ and their uniform convergence, we see that also $g_{n,M}'(\infty) \to g_{0,M}'(\infty)$.

By Lemma~\ref{comparisonSP}, since the range of $g_{\alpha,M}$ is a subset of the range of $f_{\alpha,M}$, we obtain
\[
f_{\alpha}'(\infty) \le \frac{f_{\alpha}^{-1}(i)}{g_{\alpha,M}^{-1}(i)} g_{\alpha,M}'(\infty)
\]
and since the range of $f_{\alpha}$ is a subset of the range of $g_{\alpha,M} - i \delta$, we obtain
\[
g_{\alpha,M}'(\infty) \le \frac{g_{\alpha,M}^{-1}(i+i\delta)}{f_{\alpha}^{-1}(i)} f_{\alpha}'(\infty).
\]
Thus,
\begin{equation}\label{twosidedestimateinfinity}
 \frac i{g_{\alpha,M}^{-1}(i+i\delta)} g_{\alpha,M}'(\infty) \le f_\alpha'(\infty) \le g_{\alpha,M}'(\infty).
\end{equation}
Applying these estimates to the sequence $f_n$ and taking limits as $n\to\infty$, we obtain
\[
\frac i{g_{M}^{-1}(i+i\delta)} g_{M}'(\infty) \le \liminf_{n\to\infty} f_{n}'(\infty) \le \limsup_{n\to\infty}f_{n}'(\infty) \le g_{M}'(\infty).
\]
Applying \eqref{twosidedestimateinfinity} also to $f$, we conclude
\[
\frac i{g_{M}^{-1}(i+i\delta)} f'(\infty) \le  \liminf_{n\to\infty} f_{n}'(\infty) \le \limsup_{n\to\infty} f_{n}'(\infty) \le \frac {g_{M}^{-1}(i+i\delta)}i f'(\infty).
\]
As $\delta \to 0$, $g_{M}$ converges to $f$, so $g_{M}^{-1}(i+i\delta) \to i$. This implies $f_n'(\infty) \to f'(\infty)$.

(ii)$\implies$(iii): this follows from the equicontinuity statement in Lemma~\ref{propequicontinuity}.

(iii)$\implies$(i):  If $f_n \to f$ pointwise on $\bbC_+$, then the domains $\Omega_n$ converge in kernel convergence with respect to the point $i$ to a domain $\Omega$ \cite[Theorem 1.8]{Po}. Consider a point $z = x + i y$, where $x=m\omega$ for some $m \in \mathbb{Z}^\nu$ and $y>0$. Pick $\epsilon \in (0,y/2)$ such that the square $B_\epsilon = (x-\epsilon,x+\epsilon) + i (y-\epsilon,y+\epsilon)$ is disjoint from the boundaries of $\Omega_{n}$, except possibly from their $m$-th slits. This can be done by finding $M>0$ such that $\lvert \mathfrak{m} \rvert \ge M$ implies $h_{\mathfrak{m},n} < y/2 < y -\epsilon$, and then ensuring, by the Diophantine condition, that $\epsilon$ is small enough that $\lvert \mathfrak{m}' \rvert \le M$, $\mathfrak{m}'\neq \mathfrak{m}$ implies $\lvert \mathfrak{m} \omega_{n} - \mathfrak{m}' \omega_{n} \rvert > \epsilon$. Then,
\[
B_\epsilon \setminus \Omega_n = \{ m\omega_{n} + i t\mid  y - \epsilon < t < \min(y+\epsilon,h_{m,n}) \}.
\]
Kernel convergence implies that we cannot have $\liminf h_{m,n} < y - \epsilon$ or $\limsup h_{m,n} > y+\epsilon$. Since this is true for any $y>0$ and any small enough $\epsilon>0$, it implies that $\lim_{n\to\infty} h_{m,n}$ exists.
\end{proof}

\section{Periodic Approximation and Translation Flows} \label{sectionTranslation}

Applying the above ingredients to the periodic approximants $V^{(r)}$ of the small quasiperiodic potential $V \in \mathcal{P}(\omega,\ve,\kappa_0)$ gives the following description:

\begin{prop}\label{propConthm}
Under the assumptions of Theorem~\ref{thmCombDirect}, there exists $\ve_1 = \ve_1(a_0,b_0,\kappa_0) > 0$ such that, if $\ve < \ve_1$ and $V \in \mathcal{P}(\omega,\ve,\kappa_0)$, then
\begin{enumerate}[(i)]
\item its Marchenko--Ostrovskii domain is a comb domain of the form \eqref{eqncombdomain1} with $h_m = h_{-m} \ge 0$ for all $m\in \zv \setminus \{0\}$, whose slit heights obey
\[
h_m \le \varepsilon^{1/2} \exp\left( - \frac{\kappa_0}5 \lvert m \rvert \right).
\]
\item The heights and gap sizes of $V$ obey the two-sided estimates 
\[
\frac 1{2A(m)} \gamma_m \le h_m \le 2 A(m) \gamma_m,
\]
where $A(m)$ is defined by \eqref{Amdefinition}.
\item The Marchenko--Ostrovskii map $w$ has a continuous extension to $\bbC$; for each $m\in \zv$ with $m\omega > 0$, there is a unique $\lambda_m \in [E_m^-, E_m^+]$ with $w(\lambda_m) = m\omega + ih_m$, and for $z\in \bbC_+$, the derivative $w'$ is given by the product formula
\begin{equation}\label{wderivativeproduct}
- i w'(z) = i \left( \frac 1{4(E_0 - z)} \prod_{\substack{l \in \mathbb{Z}^\nu \\ l \omega > 0}} \frac{(\lambda_{l}- z)^2}{(E_{l}^- - z)(E_{l}^+ - z)} \right)^{1/2}.
\end{equation}
\item If the $m$-th gap is open, the derivative $w'$ has an analytic extension through $(E_m^-, E_m^+)$ given by the same product formula \eqref{wderivativeproduct}, and
\[
w'(E) = O( \lvert E - E_m^\pm \rvert^{-1/2}), \qquad E \to E_m^\pm, \quad E \in (E_m^-, E_m^+).
\]
The slit height is given by the convergent integral
\[
h_m = \frac 12 \int_{E_m^-}^{E_m^+} \left\lvert  w'(E) \right\rvert dE.  
\]
\end{enumerate}
\end{prop}

We note that some qualitative parts of this statement can also be obtained in other ways by using previously proved properties of these potentials, including homogeneity of its spectrum \cite{DGL3}; for instance, existence of a product representation \eqref{wderivativeproduct} is a general consequence of the fact that the Lyapunov exponent is zero on the spectrum. We give a self-contained proof which also prepares for the more specialized quantitative estimates.

\begin{proof}
We use the sequence of periodic approximants \eqref{eqnPeriodicApproximantVr}, denoting their associated objects with the superscript $(r)$. We transform their Marchenko--Ostrovskii maps $w^{(r)}$ into the form
\[
\Theta^{(r)}(z) = - i w^{(r)}(z^2+E_{0}^{(r)}).
\]
The bottom of the spectrum converges, $E_{0}^{(r)} \to E_0$, see \cite[Theorem~7.1]{DGL2}. By Prop.~\ref{prop:wconvoffboundary}, $w^{(r)}$ converge pointwise on $\bbC_+$ to $w$, so the functions $\Theta^{(r)}$ converge uniformly on $\ol{\bbC_+}$ to the function
\[
\Theta(z) = - i w(z^2+ E_0).
\]
The derivatives $(\Theta^{(r)})'$ also converge to $\Theta'$ on $\bbC_+$. 

Next, let us establish convergence of gap edges and critical points. On the real line,  $\Re \Theta^{(r)}$ are continuous increasing functions converging to $\Re\Theta$ and they all map $0$ to $0$. Note
\[
E_m^\pm = \pm \max \{\pm E \in \bbR \mid \Re \Theta(E) = m\omega\}
\]
and analogous characterizations of $E_{m}^{\pm,(r)}$. If $E > E_m^+$, then $\Theta(E) > m\omega$, so $\Theta^{(r)}(E) > m\omega$ for all large enough $r$, so $E > E_m^{+,(r)}$ for all large enough $r$. Thus, $\limsup_{r\to\infty} E_m^{+,(r)} \le E$. Together with analogous arguments with opposite inequalities, this gives
\[
E_m^- \le \liminf_{r\to\infty} E^{-,(r)}_{m} \le \limsup_{r\to\infty} E^{+,(r)}_{m} \le E_m^+.
\]
If the $m$-th gap is closed, this implies
\[
E_m^- = \lim_{r\to\infty} E^{-,(r)}_{m} = \lim_{r\to\infty} \lambda^{(r)}_{m} =  \lim_{r\to\infty} E^{+,(r)}_{m} = E_m^+.
\]
If the $m$-th gap is open, convergence of the exponential Herglotz representations
\[
i (w^{(r)})'(z) =  \frac 1{2\sqrt{E_0 - z}} e^{\int_{[E_0^{(r)},\infty) \setminus \cS^{(r)}} \frac{\xi^{(r)}(\lambda)}{\lambda - z} \,d\lambda} 
\]
implies weak convergence of $\xi^{(r)}(\lambda)\,d\lambda$ against $C_c(\bbR)$ functions, so it implies convergence $E_m^{\pm,(r)} \to E_m^\pm$ and $\lambda_m^{(r)}$ to $\lambda_m$ for some $\lambda_m \in [E_m^-, E_m^+]$. Due to  $w(\lambda_m) = \lim_{r\to\infty} w^{(r)}(\lambda_m^{(r)})= m\omega +  ih_m$, we have $E_m^- < \lambda_m < E_m^+$.

For $z \in \bbC \setminus \cS$, the estimates
\[
\left\lvert \frac{\lambda_{m}- z}{E_{m}^\pm -  z} - 1 \right\rvert = \left\lvert \frac{\lambda_{m}- E_m^\pm}{E_{m}^\pm -  z}  \right\rvert  \le \frac{\gamma_m}{\dist(z, \cS)}
\]
give estimates uniform on compacts, and uniform in $r$, on the terms in the product representations of $w'$. Thus, convergence of each term of the product implies that the product formulas converge. The remaining claims now follow from the estimates in Lemma~\ref{lemma:productloglogm}.
\end{proof}

In particular, this contains Theorem~\ref{thmCombDirect}:

\begin{proof}[Proof of Theorem~\ref{thmCombDirect}] 
The comb domain structure is described in Prop.~\ref{propConthm}, and continuity of $h_m$ in $V$ follows from Prop.~\ref{propMOconvergence}. 
\end{proof}

We now turn to the converse direction: given a comb domain with exponentially decaying slit heights $h_m$, 
does it correspond to quasiperiodic Schr\"odinger operators in our regime? The frequency module could be read off from the locations of the bases ($x$-coordinates) of the slits \cite{EY,SY2,SY1}, but the quantitative information about the analyticity and smallness of the potential is more specialized.  We will study this by periodic approximation, building upon the result of Marchenko that in the periodic case, any $\ell_1^2$ set of slits is possible.

The proof is an adaptation of methods of \cite{DGL2}. The work \cite{DGL2} starts from the isospectral tori of a small quasiperiodic potential $V$ and its periodic approximants $V^{(r)}$, and  shows convergence of the trajectories of the corresponding Dubrovin flows. This is used to conclude that the trajectory on the isospectral torus of $V$ gives a potential of the desired form. Here, instead, we will start from a desired set of slits $h_m$. We will take periodic isospectral tori with frequencies $\omega^{(r)}$ and with essentially the same set of slits. We will show that the trajectories of the Dubrovin flows converge, that the limits are quasiperiodic operators in our regime, and that their slits are precisely $h_m$. Another application of this method is a continuity statement: within our class of potentials, we wish to know that convergence of slit heights and Dirichlet data implies convergence of potentials. We therefore point out that the approach in \cite{DGL2} is robust enough to yield both of those applications. Assume that we have a sequence of sets
\[
\mathcal{S}^{(r)} = [E_0^{(r)},\infty) \setminus \bigcup_{\substack{\mathfrak{m} \in \mathfrak{Z}(\omega^{(r)}) \\  \mathfrak{m} \omega^{(r)} > 0}} (E^{-,(r)}_\mathfrak{m} ,E^{+,(r)}_\mathfrak{m} )
\]
such that, uniformly in $r$, we have the following gap-band estimates:
\begin{enumerate}[(i)]
\item For any $\mathfrak{m}$,
\[
E_{\mathfrak{m} }^{+,(r)} - E_{\mathfrak{m} }^{-,(r)} \le \ve_1 \exp( - \kappa \lvert \mathfrak{m} \rvert ).
\]
\item For any $\mathfrak{m} ' \neq \mathfrak{m} $ with $\lvert \mathfrak{m} ' \rvert \ge \lvert \mathfrak{m} \rvert$,
\[
\dist\left( [E_\mathfrak{m}^{-,(r)},E_\mathfrak{m}^{+,(r)}],[E_{\mathfrak{m}'}^{-,(r)},E_{\mathfrak{m}'}^{+,(r)}]\right) \ge a \lvert \mathfrak{m}' \rvert^{-b} \ge \left(  \ve_1 \exp( - \kappa \lvert \mathfrak{m}'\rvert ) \right)^{1/4}.
\]
\item For any $\mathfrak{m}$ with $\mathfrak{m} \omega > 0$,
\[
a \lvert \mathfrak{m}\rvert^{-b} \le E_\mathfrak{m}^{-,(r)} - E_0^{(r)} \le C \lvert \mathfrak{m} \rvert^2.
\]
\end{enumerate}

Assume that the sets $\mathcal{S}^{(r)}$ converge to
\[
\mathcal{S} = [E_0,\infty) \setminus \bigcup_{\substack{m\in \zv \\ m\omega > 0}} (E^{-}_m,E^{+}_m)
\]
in the sense that for each $m$, $E_m^{\pm,(r)} \to E_m^\pm$, and that $\mathfrak{N}(\omega) = \{0\}$. Clearly, then, the set $\mathcal{S}$ obeys the same gap-band estimates listed above.

We will then consider Dubrovin flows corresponding to the sets $\cS^{(r)}$. Recall that the Dirichlet data $(\mu_{\mathfrak{m}}^{(r)}, \sigma_{\mathfrak{m}}^{(r)})_{\mathfrak{m}} \in \cD(\cS^{(r)})$ can be reparametrized in terms of angular variables $\varphi_{\mathfrak{m}}^{(r)}$ so that
\begin{align*}
\mu_{\mathfrak{m}}^{(r)} & = E_{\mathfrak{m}}^{-,(r)} + (E_{\mathfrak{m}}^{+,(r)} - E_{\mathfrak{m}}^{-,(r)}) \sin^2 \varphi_{\mathfrak{m}}^{(r)} \\
\sigma_{\mathfrak{m}}^{(r)} & = \sgn \sin \varphi_{\mathfrak{m}}^{(r)}
\end{align*}
and that the Dubrovin flow corresponding to translation is then expressed as
\[
\partial_x \varphi_{\mathfrak{m}}^{(r)} = \sqrt{ (\mu_{\mathfrak{m}}^{(r)} - E_0^{(r)} ) \prod_{{\mathfrak{m}}' \neq {\mathfrak{m}}}  \frac{ (E_{{\mathfrak{m}}'}^{-,(r)} - \mu_{\mathfrak{m}}^{(r)} )(E_{{\mathfrak{m}}'}^{+,(r)} - \mu_{\mathfrak{m}}^{(r)} )  }{(\mu_{{\mathfrak{m}}'}^{(r)} - \mu_{\mathfrak{m}}^{(r)})^2  }   }.
\]
The corresponding potential   $Q^{(r)}(x, \varphi_{\mathfrak{m}}^{(r)})$ is obtained by the trace formula,
\[
Q^{(r)}(x,\varphi_{\mathfrak{m}}^{(r)}) = E_0^{(r)} +  \sum_{{\mathfrak{m}}} E_{\mathfrak{m}}^{+,(r)} + E_{\mathfrak{m}}^{-,(r)} - 2 \mu_{\mathfrak{m}}^{(r)}.
\]
The following lemma follows from the arguments in  \cite[Sections 8--9]{DGL2}:

\begin{lemma}\label{propconvofdubrovinflows}
If $\mathcal{S}^{(r)}$, $\mathcal{S}$ are as above, and the initial data converge, i.e.\ $\varphi_{\mathfrak{m}}^{(r)} \to \varphi_m$ for all $m$, then the functions $Q^{(r)}(t)$ converge uniformly on compacts, i.e.\ there is a function $Q(t)$ such that for all $T< \infty$,
\[
\lim_{r\to\infty} \max_{-T \le t \le T} \lvert Q(t) - Q^{(r)}(t) \rvert = 0.
\]
\end{lemma}

Further, we characterize isospectrality in terms of comb domains and averages:

\begin{lemma} 
Let $V, Q \in \mathcal{P}(\omega,\ve,\kappa_0)$. The operators $H_V$ and $H_Q$ are isospectral if and only if
\begin{equation}\label{averagesQVequal}
\bbE(Q) = \bbE(V)
\end{equation}
and
\begin{equation}\label{hmVQequal}
h_m(Q) = h_m(V), \qquad  \forall m\in\zv.
\end{equation}
\end{lemma}

\begin{proof}
If $H_V, H_Q$ are isospectral, they have the same map $w$. From \eqref{eqn2term}, we obtain equality of averages, and from the comb domain $w(\bbC_+)$ we obtain equality of slits.

For the converse, note that \eqref{hmVQequal} implies that $V$ and $Q$ have the same Marchenko--Ostrovskii map $\Theta$. The spectrum is determined by $\Theta$ up to an additive constant (we don't yet know that they have the same minimum of the spectrum); therefore, $\cS: = \sigma(H_V) = \sigma(H_{Q-c})$ for some $c\in \mathbb{R}$. This implies that $Q-c \in \mathcal{R}(\cS)$ and, by the previous proposition, that $\bbE (Q-c) = \bbE(V)$. Comparing with \eqref{averagesQVequal}, we conclude that $c=0$.
\end{proof}

Further, we can conclude the following:

\begin{prop}
Let $V\in \mathcal{P}(\omega,\ve,\kappa_0)$ and denote $\cS = \sigma(H_V)$. Let $Q \in \mathcal{R}(\cS)$. Then $Q \in \mathcal{P}(\omega,\sqrt{2\ve},\kappa_0/2)$ and \eqref{averagesQVequal}, \eqref{hmVQequal} hold.
\end{prop}

\begin{proof}
This proposition follows from the arguments in the proof of Theorem $I$ in \cite{DGL2}. The proof given there starts from the periodic approximants $V^{(r)}(x) = U(\omega^{(r)} x)$ and constructs certain periodic potentials $Q^{(r)}(x)$ (denoted by $Q^{(r)}(x,\gamma^{(r)})$ there) which are isospectral with $V^{(r)}(x)$. This implies that $Q^{(r)}$ and $V^{(r)}$ have the same Marchenko--Ostrovskii slits $h_m^{(r)}$, since the slits are determined by the spectrum. Moreover, $Q^{(r)}$ are of the form
\[
Q^{(r)}(x) = \sum_{m\in\zv} d^{(r)}(m) e^{2\pi i m \omega^{(r)} x}
\]
and it is shown that
\[
\lvert d^{(r)}(m) \rvert \le \sqrt{2\ve} \exp\left( -\frac{\kappa_0}2 \lvert m\rvert \right), \quad m\in\zv
\]
and that $Q^{(r)}$ converge to $Q$ in the sense that $\lim_{r\to\infty} d^{(r)}(m) \to d(m)$ for each $m$ and
\[
Q(x) = \sum_{m\in\zv} d(m) e^{2\pi i m\omega x}.
\]
By Prop.~\ref{propMOconvergence}, this suffices to conclude
\[
h_m(Q) = \lim_{r\to\infty} h_m^{(r)} = h_m(V).
\]
Isospectrality of $Q^{(r)}$ with $V^{(r)}$ implies that $\bbE(Q^{(r)}) = \bbE(V^{(r)})$, so Lemma~\ref{lemmaUniformCompacts} implies \eqref{averagesQVequal}.
\end{proof}

\begin{prop} Let $\ve^{(0)}(a_0,b_0,\kappa_0)$ be as in Theorem~\ref{thmtildef}. Let $h_m$, $m\in \mathbb{Z}^\nu$, obey the conditions
\[
h_m \le \ve' \exp\left( - \kappa \lvert m\rvert \right)
\]
with $\ve' < \ve^{(0)}$ and $\kappa \ge 5 \kappa_0$. Assume that the angular Dirichlet data $\varphi_m$ are provided.

Then there exists a quasi-periodic operator $V \in \mathcal{P}(\omega,(\ve')^{1/4},\frac {\kappa}3)$ with slits $h_m(V) = h_m$ and angular Dirichlet data $\varphi_m(V) = \varphi_m$.
\end{prop}

\begin{proof}
For the rational approximants $\omega^{(r)}$, we know that this set of slits $h_m$ is allowed, i.e. there are periodic potentials $V^{(r)}$ which have the set of slits equal to $h_m$ and angular Dirichlet data $\varphi_m$. Moreover, by Theorem~\ref{thmtildef}, these potentials are of the form
\[
V^{(r)}(x) = U^{(r)}(\omega^{(r)}x), \qquad U^{(r)}(\alpha)=\sum_{m \in \mathbb{Z}^\nu} c^{(r)}(m)  \exp(2\pi i m \alpha), \qquad \lvert c^{(r)}(m) \rvert \le (\ve')^{1/3} \exp( - \frac {\kappa}5 \lvert m\rvert).
\]
By Lemma~\ref{propconvofdubrovinflows}, these potentials converge uniformly on compacts to a potential $V$.

On the other hand, if a subsequence of $U^{(r)}$ converges uniformly to a function $U$, then $V(x) = U(\omega x)$, which determines the function $U$ uniquely. Thus, every convergent subsequence of $U^{(r)}$ converges to the same limit, and by precompactness of the set of functions $U^{(r)}$ in $C(\mathbb{T}^\nu,\mathbb{R})$, we conclude that the sequence $U^{(r)}$ does in fact converge uniformly to a limit $U$.

Clearly, $U \in C(\mathbb{T}^\nu,\mathbb{R})$ is of the form
\[
U(\alpha) = \sum_{m \in \mathbb{Z}^\nu} c(m)  \exp(2\pi i m \alpha), \qquad \lvert c(m) \rvert \le (\ve')^{1/4} \exp( - \frac {\kappa}3 \lvert m\rvert).
\]
This provides us with the desired quasi-periodic operator $V(x)=U(\omega x)$ which, by Prop.~\ref{propMOconvergence}, obeys $h_m(V) = \lim_{r\to\infty} h_m(V^{(r)}) = h_m$.
 
Moreover, $V$ is given by the Dubrovin flow corresponding to the spectrum $\mathcal{S}$ determined by the slits $h_m$, with initial conditions $\varphi_m$, so we conclude that $\varphi_m(V) = \varphi_m$.
\end{proof}

\begin{proof}[Proof of Theorem~\ref{thmCombInverse}]
Since the spectrum is homogeneous, the isospectral torus is parametrized by the Dirichlet data \cite{SY2}; by the previous proposition, any $Q \in \cR(\cS)$ is in $ \mathcal{P}(\omega,(\ve')^{1/4},\frac {\kappa}3)$.
\end{proof}

We now know that for each set of exponentially decaying $h_m$ and each set of angular Dirichlet data $\varphi_m$, there is a unique potential $V$. In fact, the above methods also imply continuity of this construction: 

\begin{prop}
Consider $V \in \mathcal{P}(\omega,\ve,\kappa_0)$ and a sequence of $V_n\in \mathcal{P}(\omega,\ve,\kappa_0)$. If $\bbE(V_n) \to \bbE(V)$, $h_m(V_n) \to h_m(V)$ and $\varphi_m(V_n) \to \varphi_m(V)$ for all $m$, then $V_n \to V$.
\end{prop}

\begin{proof} By the same arguments as above, if $\bbE(V_n)$, $h_m(V_n)$ converge, then the spectra $\mathcal{S}_n$ converge; moreover, if the $\mathcal{S}_n$ converge and the $\varphi_m$ converge, then the potentials converge.
\end{proof}

\section{Gap Edges and Differentiability} \label{sectionGaps}

It is clear from general principles that gap edges are $1$-Lipshitz functions of $V$. However, for what follows, we need to establish stronger smoothness properties.

We begin by recalling the gap edge behavior of Weyl solutions \cite[Prop. 2.1]{BDGL1}. Fix $m\in \zv$, $m\omega > 0$, such that the $m$-th gap of $H_V$ is open. For $z\in (E_m^-, E_m^+)$, denote by $\psi_\pm(x;z)$ the Weyl solutions of $H_V$ at $\pm\infty$. Let $E \in \{ E_m^-, E_m^+\}$ be one of the gap edges. There exists a nontrivial formal eigensolution $\psi(x;E)$ of $H_V \psi(x;E) = E \psi(x;E)$ which is a limit of Weyl solutions from the gap, in the following sense: there exist $c_\pm(z) \in \mathbb{C}$ such that
\begin{equation}\label{gapedgeWeylconv} 
\lim_{\substack{z\to E\\ z\in (E_m^-, E_m^+)}} c_\pm(z) \psi_\pm(x;z) = \psi(x;E)
\end{equation}
uniformly in $x$ on compact subsets of $\mathbb{R}$.

Let $V(x) = U(\omega x)$ be our quasiperiodic potential, with $U: \mathbb{T}^\nu \to \mathbb{R}$ its sampling function. For $\beta\in \mathbb{T}^\nu$, let us denote by $\mu_l(\beta)$ the Dirichlet data corresponding to the potential $V_\beta(x) = U(\beta + \omega x)$; at a closed gap we will set $\mu_l = E_l^\pm$.

The following product will be repeatedly useful, for $m\in \zv$ and $z \in [E_m^-, E_m^+]$:
 \[
 \Phi_m(\beta;z) = \frac 12 \sqrt{ \frac 1{E_0 - z} \prod_{\substack{l\neq m \\ l \omega > 0}} \frac{(\mu_l(\beta) - z)^2}{(E_l^- - z)(E_l^+ - z)}}.
 \]
This combination of factors occurs in Craig's product formula for the Green's function, 
\begin{equation}\label{Craigproductformula}
G(x,x;z,V) = \frac 12 \sqrt{\frac 1{E_0 - z}  \prod_{\substack{l  \in \zv \\ l \omega > 0}} \frac{(\mu_l(\omega x) - z)^2}{(E_l^--z)(E_l^+-z)}}
\end{equation}
but there the product includes the term with index $l=m$. Therefore, the formula \eqref{Craigproductformula} can be written as
\begin{equation}\label{GreenformulaPhi}
G(x,x;z,V) = \frac{z-\mu_m(\omega x)}{\sqrt{(E_m^+ - z)(z-E_m^-)}} \Phi_m(\omega x;z).
\end{equation}
We also define a similar product involving the critical values of $w$,
\[
\chi_m(z) = \frac {1}{2\sqrt{z-E_0}}  \left( \prod_{l \neq m} \frac{(\lambda_{l}  - z)^2}{(E_{l}^- -  z)(E_{l}^+  - z)} \right)^{1/2},
\]
so that \eqref{wderivativeproduct} can be rewritten as
\[
w'(z) =  \frac{z - \lambda_{m}}{\sqrt{(E_{m}^+ - z)(z-E_m^-)}} \chi_m(z).
\]

\begin{lemma}\label{lemmaQPeigensolution}
Fix $m\in \zv$, $m\omega > 0$, and assume that the $m$-th gap is open. Let $E \in \{ E_m^-, E_m^+\}$ be one of the gap edges. The squared eigensolution $\psi^2(x;E)$ is quasiperiodic in $x$, with frequency vector $\omega$. When normalized so that
\begin{equation}\label{gapedgesolutionnormalization}
\bbE( \psi^2(x;E) ) = 1,
\end{equation}
it is represented by the product formula
\begin{equation}\label{gapedgesolutionnew}
\psi^2(x;E) = \frac { \lvert \mu_m(\omega x) - E \rvert \Phi_m(\omega x;E) }{\lvert \lambda_m - E \rvert \chi_m(E)}.
\end{equation}
\end{lemma}

\begin{proof}
Recall that
\[
G(x,x;z) = \frac{\psi_-(x;z) \psi_+(x;z)}{W(\psi_+(\cdot;z), \psi_-(\cdot;z))}.
\]
It follows from \eqref{gapedgeWeylconv} that
\[
\lim_{\substack{z\to E\\ z\in (E_m^-, E_m^+)}}  \frac{W(\psi_+(\cdot,z), \psi_-(\cdot,z))}{c_-(z) c_+(z)} G(x,x;z) = \psi^2(x;E)
\]
uniformly in $x$ on compact subsets of $\mathbb{R}$. Meanwhile, \eqref{Craigproductformula} implies that
\[
\lim_{\substack{z\to E\\ z\in (E_m^-, E_m^+)}}  \sqrt{(E_m^+-z)(z-E_m^-)} G(x,x;z) =  (\mu_m(\omega x) - E) \Phi_m(\omega x;E)
\]
converges uniformly on compacts, so the two limits must be equal up to a multiplicative constant. Thus, the squared eigensolution is of the form
\begin{equation}\label{gapedgesolution}
\psi^2(x;E) = \frac 1C \lvert \mu_m(\omega x) - E \rvert \Phi_m(\omega x;E)
\end{equation}
for some $C \in (0,\infty)$. The right hand side of \eqref{gapedgesolution} is a continuous function of $\mu(\omega x) \in \mathcal{D}(\cS)$, and the map $\mu: \mathbb{T}^\nu \to \mathcal{D}(\cS)$ is continuous, so $\psi(x;\lambda)^2$ is a quasiperiodic function with frequency $\omega$.

As for normalization, for any $z\in (E_m^-, E_m^+)$, we have
\[
\bbE( G(x,x;z)) = w'(z)
\]
(see \cite{JM}) which can be rewritten as
\[
\int_{\mathbb{T}^\nu} \frac{z-\mu_m(\beta)}{\sqrt{(E_m^+ - z)(z-E_m^-)}} \Phi_m(\beta;z) \, d\beta =  \frac{z-\lambda_m}{\sqrt{(E_m^+ - z)(z-E_m^-)}} \chi_m(z). 
\]
After cancelling the square roots in the denominator, we obtain
\[
\int_{\mathbb{T}^\nu} \lvert E-\mu_m(\beta)\rvert \Phi_m(\beta;E) d\beta =  \lvert E-\lambda_m\rvert \chi_m(E). 
\]
Both sides of this equality converge as $z\to E$, so the equality holds also for $z=E$, and the choice
\[
C = \lvert \lambda_m - E \rvert \chi_m(E)
\]
results in the normalization \eqref{gapedgesolutionnormalization}.
\end{proof}

For now on, we will fix the normalization \eqref{gapedgesolutionnew}, so that  \eqref{gapedgesolutionnormalization} holds. Note that \eqref{gapedgesolutionnew} can be rewritten as
\begin{equation}\label{gapedgesolutionnormalized}
\psi^2(x;E)  = \prod_{\substack{l \in \zv \\ l\omega > 0}} \frac{\mu_l(\omega x) - E}{\lambda_l - E}.
\end{equation}
We will also denote this by $\psi^2(x;E,V)$ when needed. Moreover, we define the set 
\[
D_m = \{ V \in\mathring\cP(\omega,\ve,\kappa_0) \mid \gamma_m(V) = 0\}.
\]
This set is closed, since $\gamma_m$ is a continuous function.

It should be kept in mind that $\psi^2(\beta;E^\pm_m,V)$ depends on $V$ through the gap edge $E^\pm_m$, through all the critical points $\lambda_l$, and through all the Dirichlet data $\mu_l$. Nonetheless, uniform control of the products implies a continuity statement: 

\begin{lemma}
If $V \in \mathring \cP(\omega,\ve,\kappa_0) \setminus D_m$, then
\begin{equation}\label{eqn30jan}
\lim_{\substack{q\to 0 \\ q \in \cB(\omega,\kappa_0)}} \bbE\left( \left\lvert \psi^2(x;E^\pm_m,V+q) - \psi^2(x;E^\pm_m,V) \right\rvert  \right) = 0.
\end{equation}
\end{lemma}

\begin{proof}
With a slight abuse of notation, we will also denote
\[
\psi^2(\beta;E_m^\pm,V)  = \prod_{\substack{l \in \zv \\ l\omega > 0}} \frac{\mu_l(\beta) - E_m^\pm}{\lambda_l - E_m^\pm}, \qquad \beta \in \bbT^\nu.
\]
For $l\neq m$, the terms of the product are estimated uniformly by
\[
\left\lvert \frac{\mu_l(\beta) - E_m^\pm}{\lambda_l - E_m^\pm} - 1 \right\rvert \le \frac{\gamma_l}{\eta_{l,m}},
\]
so pointwise convergence of Dirichlet data implies, for each $\beta \in \bbT^\nu$,
\[
\lim_{\substack{q\to 0 \\ q \in \cB(\omega,\kappa_0)}} \psi^2(\beta;E^\pm_m,V+q) = \psi^2(\beta;E^\pm_m,V).
\]
The same upper bounds give a (constant) dominating function, so by dominated convergence,
\[
\lim_{\substack{q\to 0 \\ q \in \cB(\omega,\kappa_0)}}  {\int_{\bbT^\nu}} \left\lvert \psi^2(\beta;E^\pm_m,V+q) - \psi^2(\beta;E^\pm_m,V) \right\rvert   \,d\beta = 0,
 \]
but this is precisely \eqref{eqn30jan}, rewritten as an integral over the hull.
\end{proof}

Our next result is that $E_m^\pm$ is differentiable where the $m$-th gap is open:

\begin{lemma}\label{EmpmC1}
For $m \in \zv$ with $m\omega > 0$ and any choice of $\pm$ sign, the gap edge $E_m^\pm(V):\mathring \cP(\omega,\ve,\kappa_0) \setminus D_m \to \mathbb{R}$ is a $C^1$ function with Fr\'echet derivative
\[
(\partial_V E_m^\pm)(q) = \bbE\left( \psi^2(x;E_m^\pm) q(x) \right).
\]
\end{lemma}

\begin{proof} 
By the estimates
\[
\left\lvert \bbE\left( \psi^2(x;E_m^\pm,V_1) q(x) \right) - \bbE\left( \psi^2(x;E_m^\pm,V_2) q(x) \right)  \right\rvert \le \bbE\left( \lvert \psi^2(x;E_m^\pm,V_1) -  \psi^2(x;E_m^\pm,V_2) \rvert \right) \lVert q \rVert_\infty
\]
and
\begin{equation}\label{eqnInftyNorms}
\lVert q \rVert_\infty \le \sum_{n\in\zv} \lVert q\rVert_{\infty,\kappa_0} e^{-\kappa_0 \lvert n \rvert} = C \lVert q\rVert_{\infty,\kappa_0}
\end{equation}
with an explicit constant $C = C(\kappa_0) < \infty$, the proposed derivative is, as a function of $V$, a continuous family of bounded linear functionals.

For the periodic approximants, it is known that the gap edges $E_m^\pm$ are $C^1$ on the set where the gap is open \cite{KP, PT}, so in a neighborhood of some $V$ with $\gamma_m(V) > 0$, the integral identities
\[
E_m^\pm(V+q) - E_m^\pm(V) = \int_0^1 \bbE\left( \psi^2(x; E_m^\pm(V+tq), V+tq) q \right) \,dt
\]
follow by periodic approximation, since we have uniform control of  products of the form \eqref{gapedgesolutionnormalized}. By continuity of $\psi^2(x;E)$, this implies the $C^1$ property.
\end{proof}

We note that the derivative is precisely what one would expect from perturbation theory, and in the periodic case this can be made rigorous by interpreting $E_m^\pm$ as an isolated eigenvalue of an operator with (anti)periodic boundary conditions; this approach does not extend to  the quasiperiodic case.

\begin{prop}\label{propgammam2C1}
The function $\gamma_m^2:\mathring\cP(\omega,\ve,\kappa_0) \to \mathbb{R}$ is $C^1$, with Fr\'echet  derivative 
\[
(\partial_V \gamma_m^2)(q) = \begin{cases}
2\gamma_m \bbE\left(  ( \psi^2(x;E_m^+) - \psi^2(x;E_m^-)) q(x) \right), &  V\notin D_m, \\
0, & V\in D_m.
\end{cases}
\]
\end{prop}

\begin{proof}
At points where $\gamma_m(V) > 0$, this follows from $\gamma_m(V) = E_m^+(V) - E_m^-(V)$ and Lemma~\ref{EmpmC1}.

The Lipshitz estimate $\lvert \gamma_m(V+q) - \gamma_m(V) \rvert \le \lVert q \rVert_\infty$ implies that $\partial_V \gamma_m^2 = 0$ at points $V$ where $\gamma_m(V) = 0$, which completes the proof.
\end{proof}

To prove differentiability of the gap midpoint,  we need more precise estimates of the behavior of products as the gap length goes to zero.

\begin{lemma}\label{lemmaproductlinear}
For $V \in \mathcal{P}(\omega,\ve,\kappa_0)$, fix $m\in\zv$, $m\omega > 0$, and let $f(E;V) \in [E_m^-(V), E_m^+(V)] \to \mathbb{R}$ be given by the product formula
\begin{equation}\label{e.feVproduct}
f(E;V) = \sqrt{ \frac 1{E - E_0} \prod_{l \neq m} \frac{(E - a_l)^2}{(E- E_l^-)(E-E_l^+)} }
\end{equation}
with arbitrary choices of $a_l(V) \in [E_l^-(V), E_l^+(V)]$.
\begin{enumerate}
\item There is a universal constant $C$ such that $C^{-1} \le f(E;V) \le C$ for all $V$ and all $E \in [E_m^-(V), E_m^+(V)]$. Moreover,
\[
f(E;V) = f(\tau_m) +  O(\gamma_m), \qquad \gamma_m \to 0,
\]
uniformly in $V \in \mathcal{P}(\omega,\ve,\kappa_0)$ and $E \in [E_m^-(V), E_m^+(V)]$.
\item If the functions $a_l(V)$ are continuous for all $l$, and $E = E(V)$ is a continuous function such that $E(V) \in [E_m^-(V), E_m^+(V)]$, then $f(E(V); V)$ is continuous in $V$.
\end{enumerate}
\end{lemma}

\begin{proof}
We can bound
\[
\left\lvert \frac{(E-a_l)^2}{(E-E_l^-)(E-E_l^+)} \right\rvert^{\pm 1} \le \frac{\max(\lvert E-E_l^-\rvert,\lvert E-E_l^+\rvert)}{\min(\lvert E-E_l^-\rvert,\lvert E-E_l^+\rvert)} = 1 + \frac{\gamma_l}{\eta_{m,l}}
\]
so \eqref{Amdefinition}, \eqref{productAmbound1} gives uniform bounds on $f(E)$. The logarithmic derivative of $f(E;V)$ is
\[
\frac{\partial_E f(E;V)}{f(E;V)} = \frac 12 \left( - \frac 1{E- E_0} + \sum_{l\neq m} \left( \frac 2{E-a_l} - \frac 1{E-E_l^-} - \frac 1{E-E_l^+} \right)\right)
\]
and terms of this series are bounded by
\[
\left\lvert \frac 2{E-a_l} - \frac 1{E-E_l^-} - \frac 1{E-E_l^+} \right\rvert \le \left\lvert \frac 1{E-E_l^-} - \frac 1{E-E_l^+} \right\rvert \le \frac{\gamma_l}{\eta_{m,l}^2}.
\]
This implies uniform convergence of the series, due to exponential bounds on the gaps and polynomial lower bounds on the gap distances. This completes the proof.
\end{proof}

We now consider the function
\[
b_m = \begin{cases}
2 \frac{\lambda_m - \tau_m}{\gamma_m}, & \gamma_m > 0 \\
0 & \gamma_m = 0
\end{cases}
\]
introduced to quantify how much the critical point $\lambda_m$ deviates from the gap midpoint $\tau_m$ defined in \eqref{eqnGapMidpoint}. If the $m$-th gap is small but nonzero, the rest of the infinite product for $w'$ is roughly constant along the $m$-th gap, so the critical point can be expected to be roughly in the middle of the gap; the following lemma is a rigorous version of this intuition:

\begin{lemma}\label{lemmabmC}
The function $b_m$ on $\cP(\omega,\ve,\kappa_0) \setminus D_m$ satisfies $b_m = O(\gamma_m)$ as $\gamma_m \to 0$. In particular, when extended by $b_m(V) = 0$ for $\gamma_m(V) = 0$, it is a continuous function on  $\mathcal{P}(\omega,\ve,\kappa_0)$.
\end{lemma}

\begin{proof}
Lemma~\ref{lemmaproductlinear} implies the product asymptotics
\[
\chi_m(E) = \chi_m(\tau_m) + O(\gamma_m), \quad \gamma_m \to 0.
\]
Observe that
\[
0 = w(E_m^+) - w(E_m^-) = \int_{E_m^-}^{E_m^+} w'(E)\, dE
\]
and use the change of variables $E = \tau_m + \frac {\gamma_m}2 s$ to rewrite this as
\[
0 = \frac{\gamma_m}2 \int_{-1}^1 \frac{b_m-s}{\sqrt{1-s^2}} (\chi_m(\tau_m) + O(\gamma_m)) ds  = \frac{\gamma_m}2 \chi_m(\tau_m) \int_{-1}^1  \frac{b_m-s}{\sqrt{1-s^2}} ds + O(\gamma_m^2).
\]
Computing the remaining explicit integral gives
\[
0  = \frac{\gamma_m}2 \chi_m(\tau_m) \pi b_m + O(\gamma_m^2)
\]
which implies $b_m = O(\gamma_m)$.
\end{proof}

The product formula for $\psi^2(x;E_m^\pm)$ contains the factor $\frac{\mu_m(\omega x) - E_m^\pm}{\lambda_m - E_m^\pm}$, whose denominator goes to zero as the gap closes. The previous lemma helps us to control this term in order to prove differentiability of gap midpoints:

\begin{prop}\label{proptaumC1} The function $\tau_m :\mathring\cP(\omega,\ve,\kappa_0) \to \mathbb{R}$ is $C^1$, with Fr\'echet derivative
\begin{equation}\label{taumderivative}
(\partial_V \tau_m)(q) = \begin{cases}
\bbE\left( \frac 12(\psi^2(x;E_m^+) +\psi^2(x;E_m^-)) q(x) \right), & \gamma_m > 0, \\
\bbE\left( \chi_m(\tau_m)^{-1}  \Phi_m(\omega x;\tau_m) q(x) \right), & \gamma_m =0.
\end{cases}
\end{equation}
\end{prop}

\begin{proof}
The main step is to verify that the quasiperiodic function
\[
g(x;V) = \begin{cases}
 \frac 12(\psi^2(x;E_m^+) +\psi^2(x;E_m^-)), & \gamma_m > 0, \\
 \chi_m(\tau_m)^{-1}  \Phi_m(\omega x;\tau_m), & \gamma_m =0
\end{cases}
\]
 is continuous in $V$. It is clearly continuous on the closed set $D_m$ and on its open complement, so the only nontrivial case is continuity at points $V \in D_m$ when approached from its complement.
Using
\[
\frac 2{\gamma_m} \lvert \mu_m(\beta) - E^\pm\rvert \Phi_m(\beta;E^\pm) = (1\mp \cos\varphi_m(\beta)) (\Phi_m(\beta;\tau_m) + O(\gamma_m))
\]
and
\[
\frac 2{\gamma_m} \lvert \lambda_m - E_m^\pm \rvert \chi_m(E_m^\pm) = (1\mp b_m) (\chi_m(\tau_m) + O(\gamma_m)) = \chi_m(\tau_m) + O(\gamma_m),
\]
we conclude that
\[
\psi^2(x;E_m^\pm)
= \frac{\lvert \mu_m(\beta) - E^\pm\rvert \Phi_m(\beta;E^\pm)}{\lvert \lambda_m  - E^\pm\rvert \chi_m(E^\pm)}
= \frac{(1\mp \cos\varphi_m(\beta)) \Phi_m(\beta;\tau_m)}{\chi_m(\tau_m)} + O(\gamma_m)
\]
and therefore
\[
\frac 12(\psi^2(x;E_m^-) + \psi^2(x;E_m^+)) = \frac{\Phi_m(\beta;\tau_m)}{\chi_m(\tau_m)} + O(\gamma_m)
\]
which shows continuity of $g$ at a point $V \in D_m$.

Now the rest of the proof follows by periodic approximation, analogously to the proof of Lemma~\ref{EmpmC1}, since in the periodic case this is known to be the formula for the derivative.
\end{proof}

\begin{proof}[Proof of Theorem~\ref{thmGaps}]
This result is contained in Props.~\ref{propgammam2C1} and \ref{proptaumC1}.
\end{proof}

\section{Actions} \label{sectionActions}

In this section, we study the actions $I_m$ defined by \eqref{eqnAction}. 

For $V \in \cP(\omega,\ve,\kappa_0)$, by the previous sections, $w'$ has square-root singularities at the gap edges, and $w(\lambda) \to im\omega$ as $\lambda \to E_m^\pm$, so by an integration by parts, \eqref{eqnAction} can be rewritten in the form
\[
I_m = \frac 2\pi \int_{E_m^-}^{E_m^+} \lambda w'(\lambda) \,d\lambda
\]
or the form
\begin{equation}\label{eqnImactiondefinition2}
I_m(V) = \frac {2}{\pi}  \int_{E_m^-}^{E_m^+} (\lambda - \lambda_m) w'(\lambda)\, d\lambda.
\end{equation}
From either \eqref{eqnAction} or \eqref{eqnImactiondefinition2}, we see that $I_m\ge 0$, and $I_m(V) = 0$ if and only if the $m$-th gap is closed. The first goal of this section is to prove the differentiability of actions:

\begin{prop}\label{propIm}
The function $I_m: \mathring\cP(\omega,\ve,\kappa_0) \to \mathbb{R}$ is a $C^1$ function, with Fr\'echet derivative \eqref{imderivative}.
\end{prop}

In the previous section, we have already used the diagonal Green's function. By Scharf \cite{Scharf} (see also Johnson--Moser \cite{JM}), for $z \in\mathbb{C} \setminus \cS$, $G(x,x;z,V)$ is an almost periodic function of $x$ with the same frequency module as $V$ (or a proper subset), and the product formula \eqref{Craigproductformula} is well-suited to be rewritten with respect to the hull: for each $z\in\mathbb{C} \setminus \cS$, if we define
\[
\Gamma(\beta;z,V) = \frac 12 \sqrt{\frac 1{E_0 - z}  \prod_{\substack{l  \in \zv \\ l \omega > 0}} \frac{(\mu_l(\beta) - z)^2}{(E_l^--z)(E_l^+-z)}},
\]
then $\Gamma(\cdot\,;\lambda,V)\in C(\mathbb{T}^\nu)$ and
\[
G(x,x;\lambda,V) = \Gamma(\omega x;\lambda,V),\quad \forall x\in\mathbb{R}.
\]
This agrees with changes of the potential in the hull in the sense that if $V(x) = U(\omega x)$ with $U: \mathbb{T}^\nu \to \mathbb{R}$, and if $V_\beta(x) = U(\omega x + \beta)$, then
\[
G(x,x;\lambda, V_\beta) = \Gamma(\omega x + \beta;\lambda,V), \quad \forall x\in \mathbb{R}.
\]
We will use this to rewrite the candidate derivative (the right-hand side of \eqref{imderivative}) as an integral over the hull and to prove that it is a continuous functional on $\cP(\omega,\ve,\kappa_0)$:

\begin{lemma}\label{lemmaLambdamC}
For any $V\in \mathcal{P}(\omega,\ve,\kappa_0)$, the expression
\[
\Lambda_m(q;V) = \bbE \left( q(x)  \int_{E_m^-}^{E_m^+} G(x,x;\lambda,V) \, d\lambda \right)
\]
defines a bounded linear functional  $\Lambda_{m}(\cdot\,;V): \cP(\omega,\ve,\kappa_0) \to \bbR$ and
\begin{equation}\label{Lambdambound}
\lvert \Lambda_{m} (q;V) \rvert \le A(m) \gamma_m \pi \lVert q \rVert_\infty.
\end{equation}
Moreover, the functional $\Lambda_{m}(\cdot \,; V)$ depends continuously on $V\in \cP(\omega,\ve,\kappa_0)$.
\end{lemma}

\begin{proof}
It is beneficial to rewrite the functional as an integral over $\beta \in \mathbb{T}^\nu$, instead of a spatial average over $x\in \mathbb{R}$. For this, we denote the sampling map $p \in C(\bbT^\nu)$ such that $q(x) = p(\omega x)$ for all $x\in \IR$, and write
\[
\Lambda_{m} (q;V) = \int_{\mathbb{T}^\nu} \left( \int_{E_m^-}^{E_m^+} \Gamma(\beta;\lambda,V) d\lambda \right) p(\beta) \, d\beta.
\]
Bounding the product formula \eqref{Craigproductformula} as in the proof of Lemma~\ref{lemma:productloglogm}, we conclude
\begin{equation}\label{Gammabound}
\lvert \Gamma(\beta;\lambda,V) \rvert \le A(m) \frac{\gamma_m}{\sqrt{(E_m^+ - \lambda)(\lambda - E_m^-)}}.
\end{equation}
Integrating this estimate along the gap gives
\begin{equation}\label{Gammaintbound}
\int_{E_m^-}^{E_m^+} \sup_{\beta\in\mathbb{T}^\nu} \lvert \Gamma(\beta;\lambda,V) \rvert d\lambda \le A(m) \gamma_m \pi
\end{equation}
which implies that $\Lambda_m(q;V)$ is well defined and bounded by \eqref{Lambdambound}.

Fix a sequence $\{V_n\}_{n=1}^\infty \in \mathcal{P}(\omega,\ve,\kappa_0)$ such that $V_n \to V$ uniformly on $\mathbb{R}$. Then $V_n = U_n(\omega x)$ and $U_n \to U$ uniformly on $\mathbb{T}^\nu$. To prove continuity, we separate cases.

If the $m$-th gap of $V$ is closed, apply \eqref{Lambdambound} to the $V_n$ and the universal estimate on $A(m)$ to conclude
\[
\lvert \Lambda_m(q;V_n) - \Lambda_m(q;V) \rvert = \lvert \Lambda_m(q;V_n) \rvert \le \pi C \lvert m\rvert^{C \log_2 \log_2 \lvert m\rvert}  \gamma_m(V_n) \lVert q \rVert_\infty.
\]
This shows the desired continuity at $V$, since $\gamma_m(V_n) \to \gamma_m(V) = 0$ as $n\to\infty$.

We assume from now on that $\gamma_m(V) > 0$. Since uniform convergence of potentials on compacts implies convergence of Green's functions,
\begin{equation}\label{Gammaconvpoint}
\lim_{n\to \infty} \Gamma(\beta; \lambda_n,V_n)  = \Gamma(\beta;\lambda,V)
\end{equation}
for any $\beta \in \mathbb{T}^\nu$ and any sequence $\lambda_n \to \lambda \in (E_m^-, E_m^+)$. If we make the choice of sequence
\[
\lambda_n = E_m^-(V_n) + \frac{\gamma_m(V_n)}{\gamma_m(V)} (\lambda - E_m^-(V)),
\]
then \eqref{Gammabound} provides the bound for dominated convergence to conclude
\[
\lim_{n\to\infty} \int_{E_m^-(V_n)}^{E_m^+(V_n)} \Gamma(\beta;\lambda,V_n) d\lambda = \int_{E_m^-(V)}^{E_m^+(V)} \Gamma(\beta;\lambda,V) d\lambda, \qquad \forall \beta\in \mathbb{T}^\nu.
\]
Another application of dominated convergence using the bound \eqref{Gammaintbound} implies that
\[
\lim_{n\to\infty} \int_{\mathbb{T}^\nu} \left( \int_{E_m^-(V_n)}^{E_m^+(V_n)} \Gamma(\beta;\lambda,V_n) d\lambda - \int_{E_m^-(V)}^{E_m^+(V)} \Gamma(\beta;\lambda,V) d\lambda \right) d\beta = 0,
\]
which immediately implies
\[
\lim_{n\to\infty} \sup_{q\in \mathcal{P}(\omega,\ve,\kappa_0) \setminus \{0\}} \frac { \lvert \Lambda_m(q;V_n) - \Lambda_m(q;V) \rvert}{\lVert q \rVert_\infty} = 0.
\]
By \eqref{eqnInftyNorms}, $\lVert q \rVert_\infty$ in the denominator can also be replaced by $\lVert q \rVert_{\infty,\kappa_0}$.
\end{proof}

To prove the other properties, we will use periodic approximation. For a periodic potential $\tilde V$, the action $I_{\mathfrak{m}}$ is defined by a contour integral in the Riemann surface which can be rewritten as an integral over the gap, see \cite[Proof of Theorem 7.1]{KP},
\[
I_{\mathfrak{m}}(\tilde V) = \frac 2\pi \int_{E_\mathfrak{m}^-}^{E_\mathfrak{m}^+} (-1)^{n(\mathfrak{m})-1} (\lambda - \lambda_{\mathfrak{m}}) \frac{ \Delta'(\lambda)}{\sqrt{\Delta^2(\lambda) - 4 }} d\lambda,
\]
where $n=n(\mathfrak{m})$ is given by \eqref{nofm}, $\Delta$ is the discriminant of $\tilde V$, and the branch of square root is such that the integrand is nonnegative. Using $\Delta(z) = 2 \cos(i Tw(z))$ and $w(z) = - y + i n(\mathfrak{m}) \pi$, $y>0$, for $z \in [E_\mathfrak{m}^-,E_\mathfrak{m}^+]$, we see that $\sin(iw(z)) =  i (-1)^{n(\mathfrak{m})} \sinh y$, so
\begin{align*}
\frac 2\pi \int_{E_\mathfrak{m}^-}^{E_\mathfrak{m}^+} (-1)^{n(\mathfrak{m})-1} (\lambda - \lambda_{\mathfrak{m}}) \frac{ \Delta'(\lambda)}{\sqrt{\Delta^2(\lambda) - 4 }} d\lambda 
& = \frac 2\pi \int_{E_\mathfrak{m}^-}^{E_\mathfrak{m}^+}  (\lambda - \lambda_{\mathfrak{m}} )  w'(\lambda) d\lambda
\end{align*}
and hence the standard definition in the periodic setting is equivalent to ours.

In the periodic case, $I_{\mathfrak{m}}$ is known to be analytic, with Fr\'echet derivative
\[
\partial_{\tilde V} I_{\mathfrak{m}} = \frac 2 \pi  \int_{E_\mathfrak{m}^-}^{E_\mathfrak{m}^+} (-1)^{n(\mathfrak{m})}  \frac{ \partial_{\tilde V}\Delta(\lambda)}{\sqrt{\Delta^2(\lambda) - 4 }} d\lambda  = \frac 2 {\pi}  \int_{E_\mathfrak{m}^-}^{E_\mathfrak{m}^+}  \partial_{\tilde V}w (\lambda) d\lambda.
\]
It is known that $\partial_{\tilde V} w(\lambda) = - \bbE\left( G(x,x;\lambda) q(x) \right)$ (see, e.g., \cite[Theorem 6.4]{JM}), so this implies that Prop.~\ref{propIm} holds for the periodic approximants. 

\begin{proof}[Proof of Prop.~\ref{propIm}]
The identity
\begin{equation}\label{eqnImdirectional}
I_m(V+q) - I_m(V) =  \frac 2 \pi \int_0^1 \Lambda_m(q;V+tq)\, dt
\end{equation}
holds for periodic approximants by the discussion above. Applying it to the sequence of periodic approximants $V^{(r)}$ and letting $r \to \infty$, the actions converge since the Marchenko--Ostrovskii maps are continuous, and the functionals $\Lambda_m$ converge by the same proof used to prove their continuity (Lemma~\ref{lemmaLambdamC}). Thus, the identity \eqref{eqnImdirectional} holds also for $V\in \mathring\cP(\omega,\ve,\kappa_0)$ and small enough $q\in \cB(\omega,\kappa_0)$. Since by Lemma~\ref{lemmaLambdamC}, the functional $\Lambda_m$ is continuous, the identity \eqref{eqnImdirectional} shows that $I_m$ is $C^1$ with  Fr\'echet derivative $\Lambda_m$.
\end{proof}

As an interlude, let us note that a trace formula can also be obtained by periodic approximation:

\begin{lemma} 
For any $V \in \mathcal{P}(\omega,\ve,\kappa_0)$, \eqref{eqnTraceFormula} holds.
\end{lemma}

\begin{proof}
The $T$-periodic approximants $V^{(r)}$ obey the trace formula \cite[Thm E.1]{KP}
\begin{equation}\label{eqnTraceFormulaPeriodic}
\sum_{n=1}^\infty 2 \pi \frac{n}T  I_n(V^{(r)}) = \frac 12 \bbE( (V^{(r)})^2). 
\end{equation}
The identity \eqref{eqnImactiondefinition2} implies $I_m \le 2 \gamma_m h_m$. Given existing uniform estimates on the gap sizes and slit heights of $V$ and its periodic approximants $V^{(r)}$, this provides a dominating function for the series, so from pointwise convergence  $I_{\mathfrak{m}}(V^{(r)})  \to I_m(V)$ we obtain convergence of the series. Thus, taking $r \to\infty$ in the trace formula \eqref{eqnTraceFormulaPeriodic} proves \eqref{eqnTraceFormula}.
\end{proof}

Our next goal is to characterize more precisely the behavior of the actions as a gap closes. We will show that the action goes to zero quadratically in the gap size, by showing that the expression $8 I_m / \gamma_m^2$ extends to a continuous function on $\mathring\cP(\omega,\ve,\kappa_0)$.

\begin{prop}\label{rhomC1}
The function
\[
\rho_m = \begin{cases}
\sqrt{ 8 I_m / \gamma_m^2 } & \gamma_m > 0 \\
\sqrt{ 2 \chi_m(\tau_m)} & \gamma_m = 0
\end{cases}
\]
 is a strictly positive, $C^1$ function on $\mathring\cP(\omega,\ve,\kappa_0)$.
\end{prop}

This requires a strengthening of Lemma~\ref{lemmaproductlinear}: 

\begin{lemma}\label{lemmaproductquadratic}
For $V \in \mathcal{P}(\omega,\ve,\kappa_0)$, fix $m\in\zv\setminus\{0\}$ and let $f(E;V) \in [E_m^-(V), E_m^+(V)] \to \mathbb{R}$ be given by the product formula \eqref{e.feVproduct} with arbitrary choices of $a_l(V) \in [E_l^-(V), E_l^+(V)]$.
\begin{enumerate}[(i)]
\item There is a universal constant $C$ such that $C^{-1} \le f(E;V) \le C$ and $-C \le \partial_E f(E;V) \le C$ for all $V$ and all $E \in [E_m^-(V), E_m^+(V)]$. Moreover,
\[
f(E;V) = f(\tau_m;V) + (E-\tau_m) \partial_E f(\tau_m) + O(\gamma_m^2), \qquad \gamma_m \to 0,
\]
uniformly in $V \in \mathcal{P}(\omega,\ve,\kappa_0)$ and $E \in [E_m^-(V), E_m^+(V)]$.
\item If the functions $a_l(V)$ are continuous for all $l$, and $E = E(V)$ is a continuous function such that $E(V) \in [E_m^-(V), E_m^+(V)]$, then $f(E(V); V)$ and $\partial_E f(E(V); V)$ are continuous in $V$.
\end{enumerate}
\end{lemma}

\begin{proof}
We build upon the proof of Lemma~\ref{lemmaproductlinear} and the bounds on $f(E;V)$ and its logarithmic derivative given there. By analogous estimates, we obtain a universal bound on the second logarithmic derivative of $f(E)$,
\[
\partial_E^2 \log f(E;V) = \frac{\partial_E^2 f(E;V)}{f(E;V)} - \left( \frac{\partial_E f(E;V)}{f(E;V)} \right)^2
\]
which is sufficient to complete the proof.
\end{proof}

We also need a strengthening of Lemma~\ref{lemmabmC}, which shows not just that the critical point approaches the midpoint as the $m$-th gap shrinks, but also gives the first correction term to this approach:

\begin{lemma}\label{lemmabmC1}
The function $b_m$ has asymptotic behavior
\[
b_m = 2 \frac{\lambda_m - \tau_m}{\gamma_m} = \frac{\gamma_m}4 \frac{\partial_E \chi_m(\tau_m)}{\chi_m(\tau_m)} + O(\gamma_m^2), \qquad \gamma_m \to 0.
\]
\end{lemma}

\begin{proof}
Using Lemma~\ref{lemmaproductquadratic}, we write
\[
\chi_m(E) = \chi_m(\tau_m) + (E-\tau_m) \partial_E \chi_m(\tau_m) + O(\gamma_m^2).
\]
Observe that
\[
0 = w(E_m^+) - w(E_m^-) = \int_{E_m^-}^{E_m^+} w'(E) dE
\]
and use the change of variables $E = \tau_m + \frac {\gamma_m}2 s$ to rewrite this as
\begin{align*}
0 & = \frac{\gamma_m}2 \int_{-1}^1 \frac{b_m-s}{\sqrt{1-s^2}} (\chi_m(\tau_m) +  \frac{\gamma_m}2 s \partial_E \chi_m(\tau_m) + O(\gamma_m^2)) ds \\
& = \frac{\gamma_m}2 \chi_m(\tau_m) \int_{-1}^1  \frac{b_m-s}{\sqrt{1-s^2}} ds + \frac{\gamma_m^2}4 \partial_E\chi_m(\tau_m) \int_{-1}^1  \frac{b_m-s}{\sqrt{1-s^2}} s ds  + O(\gamma_m^3)  \\
&  = \frac{\gamma_m}2 \chi_m(\tau_m) \pi b_m  - \frac{\gamma_m^2}8 \pi \partial_E\chi_m(\tau_m) + O(\gamma_m^3)
\end{align*}
which implies the lemma.
\end{proof}

\begin{proof}[Proof of Prop.~\ref{rhomC1}]
To prove continuity of $\rho_m^2$, the only nontrivial case is to prove convergence of $\rho_m^2$ along a sequence $V_n$ with $\gamma_m(V_n) > 0$ but $\gamma_m(V_n) \to 0$ as $n\to\infty$. We write
\[
I_m = \frac 2{\pi} \int_{E_m^-}^{E_m^+} (E-\lambda_m) \frac{dw}{d E} dE =  \frac 2{\pi} \int_{E_m^-}^{E_m^+} \frac{(E - \lambda_m)^2}{\sqrt{(E_m^+ - E)(E-E_m^-)}} \chi_m(E) \, dE.
\]
We use the change of variables
\begin{equation}\label{Eschange}
E = \tau_m + s \gamma_m / 2, \qquad s \in [-1,1]
\end{equation}
to write
\[
I_m =\frac 1{2\pi} \gamma_m^2 \int_{-1}^1 \frac{(s-b_m)^2}{\sqrt{1-s^2}} \chi_m(\tau_m + s \gamma_m /2) \, ds
\]
and therefore, on the complement of $D_m$,
\begin{equation}\label{rhomintegral}
\rho_m^2 = 8 \frac{I_m}{\gamma_m^2} =\frac 4{\pi} \int_{-1}^1 \frac{(s-b_m)^2}{\sqrt{1-s^2}} \chi_m(\tau_m + s \gamma_m /2) \, ds.
\end{equation}
Using $\chi_m(\tau_m+s\gamma_m/2) = \chi_m(\tau_m) + O(\gamma_m)$ and Lemma~\ref{lemmabmC} proves that
\[
\rho_m^2 = \frac 4\pi \chi_m(\tau_m) \int_{-1}^1 \frac{s^2}{\sqrt{1-s^2}} ds + O(\gamma_m) = 2 \chi_m(\tau_m) + O(\gamma_m),
\]
which proves continuity at points in $D_m$. Thus, $\rho_m^2$ is continuous.

For differentiability, we need a higher order version of the above analysis: starting from \eqref{rhomintegral} but using Lemma~\ref{lemmaproductquadratic},
\[
\chi_m(\tau_m + s \gamma_m /2 ) = \chi_m(\tau_m) + \frac {\gamma_m}2 s \partial_E \chi_m(\tau_m) + O(\gamma_m^2)
\]
gives
\begin{align*}
\rho_m^2 & = \frac 4{\pi} \int_{-1}^1 \frac{(s-b_m)^2}{\sqrt{1-s^2}} \left( \chi_m(\tau_m) + \frac {\gamma_m}2 s \partial_E \chi_m(\tau_m) + O(\gamma_m^2) \right) \, ds \\
 & = \frac 4{\pi} \int_{-1}^1 \left( \frac{s^2}{\sqrt{1-s^2}} \chi_m(\tau_m) - 2 b_m  \frac{s}{\sqrt{1-s^2}} \chi_m(\tau_m) + \frac {\gamma_m}2 \frac{s^3}{\sqrt{1-s^2}} \partial_E \chi_m(\tau_m)  + O(\gamma_m^2)\right) \, ds  \\
 & = 2 \chi_m(\tau_m) + O(\gamma_m^2),
\end{align*}
where we used
\[
\int_{-1}^1 \frac 1{\sqrt{1-s^2}} ds = \pi, \quad \int_{-1}^1 \frac s{\sqrt{1-s^2}} ds = 0, \quad \int_{-1}^1 \frac {s^2}{\sqrt{1-s^2}} ds = \frac \pi 2, \quad \int_{-1}^1 \frac {s^3}{\sqrt{1-s^2}} ds = 0.
\]
Using this, let us prove that the derivative
\begin{align*}
\partial_V \rho_m^2 & = \frac 1{\gamma_m^2} \left( 8 \partial_V I_m - \rho_m^2 \partial_V \gamma_m^2 \right) 
\end{align*}
extends continuously to $\mathcal{P}(\omega,\ve,\kappa_0)$. Using \eqref{GreenformulaPhi} and the change of variables \eqref{Eschange} and applying Lemma~\ref{lemmaproductquadratic} to $\Phi_m(\omega x;E)$ for all $x$ gives
\begin{align}
\partial_V I_m & = \frac 2\pi \int_{E_m^-}^{E_m^+} \frac{E-\mu_m(\omega x)}{\sqrt{(E_m^+ - E)(E-E_m^-)}} \Phi_m(\omega x;E) \, dE \nonumber \\
& = \frac {\gamma_m}\pi  \int_{-1}^1 \frac{s - \cos \varphi_m(\omega x)}{\sqrt{1-s^2}} \Phi_m(\omega x;\tau_m+\frac{\gamma_m}2 s) \, ds \nonumber \\
& = \frac {\gamma_m}\pi \int_{-1}^1 \frac{s-\cos \varphi_m(\omega x)}{\sqrt{1-s^2}} \Phi_m(\omega x;\tau_m) \, ds 
+ \frac{\gamma_m^2}{2 \pi} \int_{-1}^1 \frac{s-\cos \varphi_m(\omega x) }{\sqrt{1-s^2}} \partial_E \Phi_m(\omega x;\tau_m) s \, ds + O(\gamma_m^3) \nonumber \\
& = - \gamma_m \cos\varphi_m(\omega x) \Phi_m(\omega x;\tau_m) + \frac {\gamma_m^2}4 \partial_E \Phi_m(\omega x;\tau_m) + O(\gamma_m^3). \label{rhomderivative1}
\end{align}
Note that the estimate in Lemma~\ref{lemmaproductquadratic} persisted through the integration since the integrand obeys a uniform bound times a square root singularity at the edges.  Now write
\begin{align*}
\frac 2{\gamma_m} \lvert E_m^\pm - \lambda_m \rvert\chi_m(E_m^\pm) & = (1\mp b_m)(\chi_m(\tau_m) \pm \frac{\gamma_m}2 \partial_E \chi_m(\tau_m) + O(\gamma_m^2)) \\
& = \chi_m(\tau_m) \mp b_m \chi_m(\tau_m) \pm \frac{\gamma_m}2 \partial_E \chi_m(\tau_m) + O(\gamma_m^2) \\
& = \chi_m(\tau_m) \pm \frac{\gamma_m}4 \partial_E \chi_m(\tau_m) + O(\gamma_m^2)
\end{align*}
which implies that
\begin{align*}
\frac 1{ \frac 2{\gamma_m} \lvert E_m^\pm - \lambda_m \rvert\chi_m(E_m^\pm) } & = \frac 1{\chi_m(\tau_m)} \left(  1 \mp \frac{\gamma_m}4 \frac{\partial_E \chi_m(\tau_m)}{\chi_m(\tau_m)} \right) + O(\gamma_m^2).
\end{align*}
Similarly,
\begin{align}
\frac 2{\gamma_m} \lvert E_m^\pm - \mu_m(\beta) \rvert \Phi_m(\beta;E_m^\pm) & = (1 \mp \cos \varphi_m(\beta) ) \left( \Phi_m(\beta;\tau_m) \pm \frac{\gamma_m}2 \partial_E \Phi_m(\beta;\tau_m) + O(\gamma_m^2) \right) . \label{product30}
\end{align}
Multiplying these two equalities,
\begin{align*}
\frac {\lvert E_m^\pm - \mu_m(\beta) \rvert \Phi_m(\beta;E_m^\pm)}{\lvert E_m^\pm - \lambda_m \rvert\chi_m(E_m^\pm) }  & =  \frac 1{\chi_m(\tau_m)} \Biggl( (1 \mp \cos \varphi_m(\beta) ) \left( \Phi_m(\beta;\tau_m) \pm \frac{\gamma_m}2 \partial_E \Phi_m(\beta;\tau_m) \right) \\
& \qquad \mp \frac{\gamma_m}4 \frac{\partial_E \chi_m(\tau_m)}{\chi_m(\tau_m)} \left( 1 \mp \cos\varphi_m(\beta)  \right)\Phi_m(\beta;\tau_m) \Biggr)  + O(\gamma_m^2).
\end{align*}
Subtracting this with $+$ minus with $-$ gives
\begin{align*}
& \frac {\lvert E_m^+ - \mu_m(\beta) \rvert \Phi_m(\beta;E_m^+)}{\lvert E_m^+ - \lambda_m \rvert\chi_m(E_m^+) } - \frac {\lvert E_m^- - \mu_m(\beta) \rvert \Phi_m(\beta;E_m^-)}{\lvert E_m^- - \lambda_m \rvert\chi_m(E_m^-) }  \\
& =  \frac 2{\chi_m(\tau_m)} \Biggl( - \cos\varphi_m(\beta) \Phi_m(\beta;\tau_m) + \frac{\gamma_m}2 \partial_E \Phi_m(\beta;\tau_m)  - \frac{\gamma_m}4 \frac{\partial_E \chi_m(\tau_m)}{\chi_m(\tau_m)} \Phi_m(\beta;\tau_m) \Biggr) + O(\gamma_m^2).
\end{align*}
Therefore
\begin{equation}\label{rhomderivative2}
\rho_m^2 2 \gamma_m(\partial_V E_m^+ - \partial_V E_m^-) = 8 \gamma_m \Biggl( - \cos\varphi_m(\beta) \Phi_m(\beta;\tau_m) + \frac{\gamma_m}2 \partial_E \Phi_m(\beta;\tau_m)  - \frac{\gamma_m}4 \frac{\partial_E \chi_m(\tau_m)}{\chi_m(\tau_m)} \Phi_m(\beta;\tau_m) \Biggr) + O(\gamma_m^3).
\end{equation}

Subtracting \eqref{rhomderivative2} from $8$ times \eqref{rhomderivative1} gives
\[
8 \partial_V I_m - \rho_m^2 \partial_V \gamma_m^2 =  - 2 \gamma_m^2 \partial_E \Phi_m(\omega x;\tau_m) + 2 \gamma_m^2 \frac{\partial_E \chi_m(\tau_m)}{\chi_m(\tau_m)} \Phi_m(\beta;\tau_m)  + O(\gamma_m^3)
\]
and dividing by $\gamma_m^2$ shows that
\[
\partial_V \rho_m^2 =  - 2 \partial_E \Phi_m(\omega x;\tau_m) + 2 \frac{\partial_E \chi_m(\tau_m)}{\chi_m(\tau_m)} \Phi_m(\omega x;\tau_m)  + O(\gamma_m),
\]
so $\partial_V \rho_m^2$ extends continuously to all of $\mathcal{P}(\omega,\ve,\kappa_0)$.
\end{proof}

\begin{proof}[Proof of Theorem~\ref{thmActions}]
This is contained in the results proved in this section.
\end{proof}

\bibliographystyle{amsplain}

\providecommand{\bysame}{\leavevmode\hbox to3em{\hrulefill}\thinspace}
\providecommand{\MR}{\relax\ifhmode\unskip\space\fi MR }
\providecommand{\MRhref}[2]{%
  \href{http://www.ams.org/mathscinet-getitem?mr=#1}{#2}
}
\providecommand{\href}[2]{#2}

\end{document}